\RequirePackage{plautopatch}
\documentclass{amsart}
\usepackage{mathtools}
\usepackage{amsmath}
\usepackage{amsthm}
\usepackage{enumitem}
\usepackage{tabularx}
\usepackage{autobreak}
\usepackage{cleveref}
\usepackage{xr}
\theoremstyle{plain}
\newtheorem{theorem}{Theorem}[section]
\newtheorem{proposition}[theorem]{Proposition}
\newtheorem{corollary}[theorem]{Corollary}
\newtheorem{lemma}[theorem]{Lemma}
\newtheorem*{theorem*}{Theorem}

\theoremstyle{definition}
\newtheorem{definition}[theorem]{Definition}

\theoremstyle{remark}
\newtheorem{remark}[theorem]{Remark}
\newtheorem*{remark*}{Remark}
\newtheorem*{acknowledgment}{Acknowledgment}

\newcommand{\R}{\mathbb{R}}
\newcommand{\N}{\mathbb{N}}
\newcommand{\net}{\mathcal{N}}
\newcommand{\D}{\mathcal{D}}
\newcommand{\X}{\mathcal{X}}
\newcommand{\F}{\mathcal{F}}

\newcommand{\M}{\mathcal{M}}
\newcommand{\DM}{\mathcal{DM}}
\renewcommand{\L}{\mathcal{L}}
\newcommand{\T}{\mathcal{T}}
\newcommand{\B}{\mathcal{B}}
\newcommand{\TB}{\mathcal{TB}}
\renewcommand{\P}{\mathcal{P}}
\newcommand{\Q}{\mathcal{Q}}
\newcommand{\E}{\mathcal{E}}
\newcommand{\G}{\mathcal{G}}
\newcommand{\s}{\mathcal{S}}

\newcommand{\su}{\mathcal{U}}
\DeclareMathOperator{\lip1}{Lip_1}
\DeclareMathOperator{\pr}{pr}
\DeclareMathOperator{\id}{id}
\newcommand{\kf}{d_{\operatorname{KF}}}
\DeclareMathOperator{\pd}{PartDiam}
\DeclareMathOperator{\od}{ObsDiam}
\newcommand{\dpr}{d_{\operatorname{P}}}
\DeclareMathOperator{\diam}{diam}
\DeclareMathOperator{\supp}{supp}

\DeclareMathOperator{\bdd}{Bdd}
\DeclareMathOperator{\cov}{Cov}
\DeclareMathOperator{\capa}{Cap}
\newcommand{\lm}{\mathrm{Lm}}
\newcommand{\dconc}{d_{\operatorname{conc}}}
\newcommand{\haus}[1]{{{#1}_{\operatorname{H}}}}
\newcommand{\hausp}[1]{{\left(#1\right)_{\operatorname{H}}}}
\newcommand{\dinf}[1]{{d^{#1}_\infty}}
\newcommand{\hauspdinf}[1]{{\hausp{\dinf{#1}}}}
\newcommand{\mmid}{\mathrel{}\middle\vert\mathrel{}}
\renewcommand{\epsilon}{\varepsilon}
\renewcommand{\phi}{\varphi}
\newcommand{\ncirc}{\mathbin{\mkern-3mu\circ\mkern-3mu}}

\title[Geometry of Geometric Data Set II: Pyramid]{Geometry of Geometric Data Set II: Pyramid}
\author{Shigeaki Yokota}
\date{}
\thanks{This work was partially supported by JSPS KAKENHI Grant Number 22H04942, for which the author served as a research assistant.}

\begin{document}
\begin{abstract}
    The observable distance $\dconc$ based on measure concentration and the box distance $\Box$ based on collapsing theory are extended to geometric data sets introduced by Hanika--Schneider--Stumme. On the set $\D$ of isomorphism classes of geometric data sets, $\dconc$ is non-separable and $\Box$ is complete and non-separable. We introduce the class $\D/\L$ of $\L$-compact geometric data sets in $\D$, for a monoidal subfamily $\L$ of 1-Lipschitz functions $\lip1(\R)$, and prove its $\Box$-completeness and separability. We then construct a natural compactification of $(\D/\L, \dconc)$ by means of \emph{$\L$-pyramids} when $\L$ contains the clipping family. We further prove a complete limit formula for the observable diameter of $\lip1(\R)$-pyramids, and show that applying our construction to Hanika--Schneider--Stumme's embedding is compatible with the compactification and preserves the polynomial-time computability of the observable diameter.
\end{abstract}
\maketitle

\section{Introduction}
Gromov~\cite{gromov2007met} developed the geometry of mm-spaces, based on the concentration of measure phenomenon and the theory of collapsing manifolds. A triple $(X, d_X, \mu_X)$, or simply $X$, is an \emph{mm-space} if $d_X$ is a complete separable metric on $X$, and $\mu_X$ is a Borel probability measure with full support on $(X, d_X)$. He defined the observable distance $\dconc$ based on measure concentration and the box distance $\Box$ based on collapsing theory on the set of all isomorphism classes of mm-spaces, say $\X$, and constructed a natural compactification of $(\X, \dconc)$, each element of which is called a \emph{pyramid}. Various properties of these distances and spaces are known, with particular attention given to the separability of $(\X, \dconc)$ and the completeness and separability of $(\X, \Box)$.

Pestov~\cite{pestov2008gds} treated data as mm-spaces and used the \emph{observable diameter} (see \Cref{def:obsdiam}) to explain a form of the curse of dimensionality through the concentration of measure phenomenon. However, an efficient algorithm for computing the observable diameter of an mm-space $X$ is not yet known, and straightforward computation using the set of all 1-Lipschitz continuous functions on $X$, say $\lip1(X)$, directly following the definition, takes exponential time with respect to the size $\#X$ of the data sample $X$. He proposed restricting the function family used to compute the observable diameter.

Following this, Hanika, Schneider, and Stumme~\cite{hanika2022gds} defined the geometric data set as a generalization of mm-spaces in order to quantify the curse of dimensionality through the observable diameter, as follows: A triple $(X, F_X, \mu_X)$, or simply $X$, is defined to be a \emph{geometric data set} if $F_X$ is a non-empty set of real-valued functions on $X$ such that
\begin{equation*}
    d_{F_X}(x, y) \coloneqq \sup_{f \in F_X} |f(x) - f(y)|,\ x,y\in X,
\end{equation*}
is a complete separable metric, and if $\mu_X$ is a Borel probability measure with full support on $(X, d_{F_X})$. An mm-space can be interpreted as a special case where $F_X = \lip1(X, d_X)$. For a geometric data set $X$, we write $d_X$ for $d_{F_X}$. We denote by $\mathcal{D}$ the set of isomorphism classes of geometric data sets. They extended the observable distance between mm-spaces to geometric data sets and similarly generalized the observable diameter to geometric data sets. They then confirmed the characteristics of the observable diameter of geometric data sets as a statistical measure.

In our previous paper~\cite{gds1}, we proved the \emph{non}-separability of $(\D, \dconc)$ and extended the box distance $\Box$ to $\D$, establishing its completeness and \emph{non}-separability. In the present paper, we continue this study by constructing a compactification of a suitable subclass of $\D$ equipped with the observable distance.

We first focus on a $\Box$-separable subset of $\D$. Fix a monoidal subfamily $\L \subset \lip1(\R)$, that is, a subfamily containing $\id_\R$ and closed under composition and pointwise convergence. We say that $X \in \D$ is \emph{$\L$-compact} if the pointwise closure $\overline{F_X}$ of $F_X$ is closed under left composition by elements of $\L$ and the $(\epsilon, \L)$-covering number
\begin{equation*}
    \cov(\overline{F_X}, \epsilon; \L) \coloneqq \inf\bigl\{\,\#\net \bigm| \net \subset \overline{F_X} \subset U(\L \circ \net, \epsilon; \kf^X)\,\bigr\}
\end{equation*}
is finite for every $\epsilon > 0$, where $\kf^X$ denotes the Ky Fan metric on Borel measurable functions on $X$. Letting $\D/\L$ denote the class of all isomorphism classes of $\L$-compact geometric data sets, we prove the following main theorem:
\begin{theorem}[\Cref{thm:DLIsBoxCompleteSeparable}]
    The $\L$-compact class $\D/\L$ is $\Box$-complete and separable.
\end{theorem}
When $\L$ contains the translation family
\begin{equation*}
    \T \coloneqq \{\, x \mapsto x + c \mid c \in \R \,\},
\end{equation*}
the theory simplifies considerably: every $\L$-closed geometric data set is automatically $\L$-compact (see \Cref{theorem:LClosedIsLCompact}), and $\L$ is self-compact (see \Cref{prop:TSubsetImpliesSelfCompact}), so $\L$-compactness can be equivalently characterized as the $\hausp{\kf^X}$-compactness of $\L \circ \F(\overline{F_X}) \coloneqq \{\L \circ F \mid F \in \F(\overline{F_X})\}$ (see \Cref{prop:SelfCompactCharacterization}), where $\hausp{\kf^X}$ is the Hausdorff distance with respect to the Ky Fan metric on $X$.

We further assume that $\L$ contains the clipping family
\begin{equation*}
    \TB \coloneqq \{\, x \mapsto l \lor (x + c) \land u \mid c \in \R,\, l \in [-\infty, +\infty),\, u \in (-\infty, +\infty],\, l \leq u \,\},
\end{equation*}
where $a \land b\coloneqq\min\{a,b\}$ and $a\lor b\coloneqq\max\{a,b\}$. Analogously to the pyramid of mm-spaces, we say that $\P \subset \D/\L$ is an \emph{$\L$-pyramid} if it satisfies the following conditions:
\begin{enumerate}[label=(\arabic*)]
    \item for any $X \in \D/\L$, if $X \preceq Y \in \P$, then $X \in \P$,
    \item for any $X, Y \in \P$, there exists $Z \in \P$ such that $X \preceq Z$ and $Y \preceq Z$,
    \item $\P$ is nonempty and $\Box$-closed.
\end{enumerate}
We then prove the following:
\begin{proposition}[Pyramid metric and compactification; \Cref{prop:PyramidMetric}]
    There exists a metric $\rho$ on $\Pi_\L$ such that the map $\iota_\L \colon (\D/\L, \dconc) \to (\Pi_\L, \rho)$ defined by $\iota_\L(X) \coloneqq \{Y \in \D/\L \mid Y \preceq X\}$ is a $1$-Lipschitz embedding, and $(\Pi_\L, \rho)$ is a compactification of $(\D/\L, \dconc)$.
\end{proposition}

Hanika et al.\ \cite{hanika2022gds} further proposed another embedding of mm-spaces, namely
\begin{equation*}
    X_\circ \coloneqq (X, \{x \mapsto d_X(x, y) \mid y \in X\}, \mu_X).
\end{equation*}
The time complexity of computing the observable diameter of $X_\circ$ is $O(\#X^3)$~\cite[Definition~3.2, \S{}6.1.1]{hanika2022gds}, which is polynomial-time and substantially more tractable than the naive exponential-time computation.

Applying our construction to Hanika--Schneider--Stumme's embedding, one can embed mm-spaces into our pyramid space while preserving this time complexity (see \Cref{prop:OdUnderHSS}). The resulting geometric data set
\begin{equation*}
    X^{\lip1(\R)}_\circ \coloneqq (X, \{x \mapsto p(d_X(x, y)) \mid y \in X,\, p \in \lip1(\R)\}, \mu_X) \in \D/\lip1(\R)
\end{equation*}
admits even a complete limit formula for the observable diameter in the embedding space:
\begin{corollary}[\Cref{cor:Lip1OdiamLimit}]
    If a sequence of $\lip1(\R)$-pyramids $\P_n$ converges weakly to $\P$, then for any $\kappa \in (0,1)$,
    \begin{align*}
        \od(\P; -\kappa) &= \lim_{\epsilon \to +0} \liminf_{n \to \infty} \od(\P_n; -(\kappa + \epsilon))\\
        & = \lim_{\epsilon \to +0} \limsup_{n \to \infty} \od(\P_n; -(\kappa + \epsilon)).
    \end{align*}
\end{corollary}
The corollary extends the limit formula for the observable diameter on mm-spaces, proposed by Ozawa--Shioya \cite{ozawa2015limit} and completely proved in \cite{yokota2024obsdiam}, to the setting of $\lip1(\R)$-pyramids.

The assumption $\TB \subset \L$ in \Cref{prop:PyramidMetric} is essential: when $\L = \T$, the weak limit of a sequence of $\L$-pyramids may be empty (see \Cref{rem:TBNecessary}), so $\Pi_\L$ fails to be compact. To handle the intermediate case $\T \subset \L$, we introduce in Section~6 the space $\Sigma$ of \emph{staircases}, which is always compact and admits a topological embedding of $(\D/\L, \dconc)$ under the assumption $\T \subset \L$ alone. Whether this embedding is dense remains open. The following table summarizes the structure of the theory.

\begin{center}
\renewcommand{\arraystretch}{1.3}
\begin{tabularx}{\textwidth}{ll>{\raggedright\arraybackslash}X}
    \hline
    Assumption on $\L$ & Examples & Result \\
    \hline
    monoidal subfamily & $\T$, $\TB$, $\lip1(\R)$ & $(\D/\L, \Box)$ is complete and separable \\
    $\T \subset \L$ & $\T$, $\TB$, $\lip1(\R)$ & $(\D/\L, \dconc)$ embeds topologically into $(\Sigma, d_\Sigma)$ (compact; density unknown) \\
    $\TB \subset \L$ & $\TB$, $\lip1(\R)$ & $(\Pi_\L, \rho)$ is a compactification of $(\D/\L, \dconc)$ \\
    $\L = \lip1(\R)$ & $\lip1(\R)$ & Complete limit formula for $\od$ of $\L$-pyramids \\
    \hline
\end{tabularx}
\end{center}

The paper is organized as follows. In Section~2, we fix notation and recall basic facts used throughout, including the Ky Fan metric, domination of geometric data sets, weak Hausdorff convergence of closed sets, and the box distance~$\Box$ on~$\D$ introduced in~\cite{gds1}. In Section~3, we introduce $\L$-compactness and the notion of self-compactness for $\L$, and establish the completeness of $\D/\L$ with respect to~$\Box$. In Section~4, we prove the $\Box$-separability of $\D/\L$ using the notion of $N$-measurements. In Section~5, we define $\L$-pyramids and study structural properties of the space $\Pi_\L$ of all $\L$-pyramids. In Section~6, we introduce staircases and the key notion of extractability, and prove that $(\D/\L, \dconc)$ admits a topological embedding into the space $\Sigma$ of staircases. In Section~7, we prove the main result: when $\TB \subset \L$, the space $(\Pi_\L, \rho)$ equipped with the pyramid metric provides a compactification of~$(\D/\L, \dconc)$. Moreover, we prove a limit formula for the observable diameter of $\lip1(\R)$-pyramids. In Section~8, we apply these results to mm-spaces via Hanika--Schneider--Stumme's embedding~$X \mapsto \D^{\lip1(\R)}_\circ(X)$, and show that the compactification preserves the observable diameter.

\section{Preliminaries}
Let $(X,d)$ be a metric space and let $r\geq 0$ be a real number. For a point $x\in X$, we denote the open ball centered at $x$ with radius $r$ by $U(x,r;d)$, or simply $U(x,r)$. Similarly, we denote the closed ball centered at $x$ with radius $r$ by $B(x,r;d)$, or simply $B(x,r)$. For a subset $A\subset X$, we denote the open $r$-neighborhood of $A$ by $U(A,r;d)$, or simply $U(A,r)$, and write $d(x,A) \coloneqq \inf_{a\in A} d(x,a)$. We write the Hausdorff distance induced by $d$ as $\hausp{d}$.

For metric spaces $(X,d_X)$ and $(Y,d_Y)$, we denote by $\lip1(X,Y)\coloneqq\{f\colon X\to Y\mid d_Y(f(x),f(x'))\leq d_X(x,x')\text{ for all }x,x'\in X\}$ the set of $1$-Lipschitz maps from $X$ to $Y$, and write $\lip1(X)\coloneqq\lip1(X,\R)$. For a product space $\prod_n X_n$, we write $\pr_n$ for the projection onto the $n$-th factor and $\pr_{n_1,\ldots,n_m}\coloneqq(\pr_{n_1},\ldots,\pr_{n_m})$.

For a set $S$ and maps $f,g\colon S\to\R$, we write $\dinf{S}(f,g)\coloneqq\sup_{x\in S}|f(x)-g(x)|$ and denote by $\hauspdinf{S}$ the Hausdorff distance with respect to $\dinf{S}$.

\begin{definition}[Ky Fan metric]
    Let $(X,\mu)$ be a measure space and $Y$ a metric space. We define the \emph{Ky Fan metric} $\kf^\mu$ on the set of $\mu$-measurable maps from $X$ to $Y$ by
    \begin{equation*}
        \kf^\mu(f,g) \coloneqq \inf\left\{\epsilon \geq 0\mid \mu(\left\{x\in X\mid d_Y(f(x),g(x)) > \epsilon\right\}) \leq \epsilon\right\}
    \end{equation*}
    for any two $\mu$-measurable maps $f$ and $g\colon X\to Y$. We often write this as $\kf^X$ or simply $\kf$.
\end{definition}

\begin{lemma}\label{lemma:1LipPullbackReductionKyFanDistance}
    Let $X$ be a geometric data set. For any $f,g\in F_X$ and $p\in\lip1(\R)$, we have
    \begin{equation*}
        \kf^X(p\ncirc f,p\ncirc g) \leq \kf^X(f,g).
    \end{equation*}
\end{lemma}
\begin{proof}
    Note that $|p\circ f(x)-p\circ g(x)| \leq |f(x)-g(x)|$ for any $x\in X$. For any real number $\epsilon > \kf^X(f,g)$, we have
    \begin{align*}
        \mu_X(\left\{x\in X\mmid |p\circ f(x)-p\circ g(x)| > \epsilon\right\}) & \leq \mu_X(\left\{x\in X\mmid |f(x)-g(x)| > \epsilon\right\}) \leq \epsilon.
    \end{align*}
    This completes the proof.
\end{proof}

\begin{lemma}\label{lemma:1LipPullbackReductionDinf}
    Let $S$ be a set and let $f,g\colon S\to\R$ be maps. For any $p\in\lip1(\R)$, we have
    \begin{equation*}
        \dinf{S}(p\circ f,\, p\circ g) \leq \dinf{S}(f,g).
    \end{equation*}
    In particular, $\hauspdinf{S}(p\circ A,\, p\circ B) \leq \hauspdinf{S}(A,B)$ for any families $A$ and $B$ of maps from $S$ to $\R$.
\end{lemma}
\begin{proof}
    For any $x\in S$, we have $|p(f(x))-p(g(x))| \leq |f(x)-g(x)|$, so $\dinf{S}(p\circ f, p\circ g) \leq \dinf{S}(f,g)$. The statement for families follows immediately.
\end{proof}

Let $X$ be a non-empty set and $Y$ a metric space. For $F\subset Y^X$, we denote by $\overline{F}$ the closure of $F$ with respect to pointwise convergence. By \Cref{lemma:pointwiseConvergenceIsMeasureConvergence} below, when $X$ is a geometric data set and $F\subset\lip1(X,\R)$, this coincides with the closure in the topology of convergence in measure.

\begin{definition}[Prohorov distance]
    Let $\mu$ and $\nu$ be two Borel probability measures on a separable metric space $X$. The \emph{Prohorov distance} between $\mu$ and $\nu$ is defined by
    \begin{equation*}
        \dpr(\mu,\nu) \coloneqq \inf\bigl\{\,\epsilon \geq 0 \bigm| \mu\bigl(U(A,\epsilon)\bigr)\geq \nu(A)-\epsilon \ \text{for any Borel subset}\ A\subset X\,\bigr\}.
    \end{equation*}
\end{definition}

\begin{lemma}[{\cite[\Cref{1:prop:pointwiseEqKf}]{gds1}}]\label{lemma:pointwiseConvergenceIsMeasureConvergence}
    Let $X$ be a geometric data set. The topology of pointwise convergence coincides with that of convergence in measure.
\end{lemma}

\begin{lemma}[{\cite[\Cref{1:lem:bddIsCompact}]{gds1}}]\label{lemma:BddIsCompact}
    Let $X$ be a geometric data set, $x\in X$ a point, and $c\geq 0$ a real number. The subfamily
    \begin{equation*}
        \bdd(X,x,c) \coloneqq \bigl\{\, f\in \overline{F_X} \,\big|\, |f(x)| \leq c \, \bigr\}
    \end{equation*}
    is compact with respect to pointwise convergence and convergence in measure.
\end{lemma}

\begin{lemma}\label{lemma:FXbarBallIsCompact}
    Let $X$ be a geometric data set. For any $f\in F_X$ and $r\in(0,1)$, the set $\overline{F_X}\cap B(f,r;\kf^X)$ is compact. In particular, $(\overline{F_X},\kf^X)$ is complete.
\end{lemma}
\begin{proof}
    The separability of $X$ implies that there exists a finite subset $\net\subset X$ such that $\mu_X\!\left(U(\net,r;d_X)\right) \geq r$. Take any $g\in \overline{F_X}\cap B(f,r;\kf^X)$. Since
    \begin{equation*}
        \mu_X\!\left(\bigl\{x\in X \,\big|\, |f(x)-g(x)| \leq r\bigr\}\right) > 1-r,
    \end{equation*}
    there exist points $x\in\net$ and $y\in U(x,r;d_X)$ such that $|f(y)-g(y)| \leq r$, so that
    \begin{align*}
        |f(x) - g(x)| & \leq |f(x) - f(y)| + |f(y) - g(y)| + |g(y) - g(x)| \\
                      & \leq d_X(x,y) + r + d_X(y,x) \leq 3r.
    \end{align*}
    Since $g$ is arbitrary, we have
    \begin{equation*}
        \overline{F_X}\cap B(f,r;\kf^X) \subset \bigcup_{x\in\net} \bdd(X,x,|f(x)| + 3r),
    \end{equation*}
    which implies that $\overline{F_X}\cap B(f,r;\kf^X)$ is $\kf^X$-precompact. Since $\kf^X$ metrizes convergence in measure and the limit of any $\kf^X$-Cauchy sequence in $\overline{F_X}$ belongs to $\overline{F_X}$, the space $(\overline{F_X}, \kf^X)$ is complete; hence $\overline{F_X}\cap B(f,r;\kf^X)$ is complete. Consequently, it is compact. This completes the proof.
\end{proof}

\begin{lemma}[{\cite[\Cref{1:lem:dInftyUniformlyContinuous}]{gds1}}]\label{lem:dInftyUniformlyContinuous}
    Let $X$ be a geometric data set, $K$ a compact subset of $X$, and $Y$ a metric space. Then, $d^K_\infty$ is uniformly continuous on $\lip1(X, Y)^2$ equipped with the $\ell^\infty$-product metric induced by $\kf^X$.
\end{lemma}

\begin{definition}[Coupling]\label{def:Coupling}
    Let $(X,\mu),(Y,\nu)$ be Borel probability measure spaces. A Borel probability measure $\pi$ on $X\times Y$ is called a \emph{coupling} of $\mu$ and $\nu$ if $(\pr_1)_*\pi = \mu$ and $(\pr_2)_*\pi = \nu$. We denote by $\Pi(\mu,\nu)$ the set of couplings of $\mu$ and $\nu$.
\end{definition}

For a metric space $(X,d)$, we denote by $\F(X,d)$ the set of all nonempty closed subsets of $X$ with respect to $d$. When the metric is clear from the context, we simply write $\F(X)$.

\begin{definition}[Weak Hausdorff convergence]\label{def:WeakHausdorffConvergence}
    Let $(X, d_X)$ be a metric space and $S_n$, $S$ be closed sets in $X$, $n = 1,2,\ldots$~. We say that $S_n$ \emph{converges to $S$ in the weak Hausdorff sense} as $n\to\infty$ if the following (1) and (2) are satisfied.
    \begin{enumerate}
        \item For any $x\in S$, we have $\lim_{n\to\infty} d_X(x,S_n) = 0$.
        \item For any $x\in X\setminus S$, we have $\liminf_{n\to\infty} d_X(x,S_n) > 0$.
    \end{enumerate}
\end{definition}

An \emph{extraction} is a strictly increasing map $\iota\colon\{1,2,\ldots\}\to\{1,2,\ldots\}$; we call $\{x_{\iota(n)}\}$ an \emph{extracted subsequence} of $\{x_n\}$.

\begin{lemma}[{\cite[Theorem 5.2.12]{beer1993topologies}}]\label{lem:weakHausdorffSubSequence}
    Any sequence of closed sets in a complete separable metric space has an extraction that converges in the weak Hausdorff sense.
\end{lemma}

\begin{lemma}[{\cite[Lemma 3.9]{nakajima2022coupling}}]\label{lem:weakHausdorffMeasureBound}
    Let $X$ be a metric space, $\mu$ be a Borel probability measure on $X$, and $S_n$, $S$ be closed sets in $X$, $n = 1,2,\ldots$~. If $S_n$ converges to $S$ in the weak Hausdorff sense, then
    \begin{equation*}
        \mu(S) \geq \limsup_{n\to\infty} \mu(S_n).
    \end{equation*}
\end{lemma}

\subsection{Observable Distance}

\begin{definition}[Geometric data set, {\cite[Definition 3.2]{hanika2022gds}}]\label{def:GDS}
    A triple $(X, F_X, \mu_X)$, or simply $X$, is a \emph{geometric data set} if $F_X$ is a non-empty set of real-valued functions on $X$ such that
    \begin{equation*}
        d_{F_X}(x, y) \coloneqq \sup_{f \in F_X} |f(x) - f(y)|
    \end{equation*}
    is a complete separable metric on $X$, and $\mu_X$ is a Borel probability measure with full support on $(X, d_{F_X})$. We write $d_X$ for $d_{F_X}$ and denote by $\D$ the set of isomorphism classes of geometric data sets.
\end{definition}

\begin{definition}[Feature order, domination, {\cite[Definition 3.4]{hanika2022gds}}]\label{def:FeatureOrder}
    We say that a geometric data set $X$ \emph{dominates} a geometric data set $Y$, written $Y\preceq X$, if there exists a Borel measurable map $f\colon X\to Y$ such that $f_*\mu_X = \mu_Y$ and $F_Y\circ f \subset \overline{F_X}$. Such an $f$ is called a \emph{domination}. The relation $\preceq$ is called the \emph{feature order relation}.
\end{definition}

\begin{definition}[Quotient geometric data set, quotient domination]\label{def:Quotient}
    Let $X$ be a geometric data set and $G\subset\overline{F_X}$. A geometric data set $Y$ is called a \emph{quotient geometric data set} of $X$ by $G$ if there exists a domination $\phi\colon X\to Y$ such that the following conditions (1) and (2) hold.
    \begin{enumerate}
        \item The equality $F_Y\circ\phi = G$ holds.
        \item For any geometric data set $Z$ and any domination $\psi\colon X\to Z$, if $F_Z\circ\psi\subset\overline{G}$, then there exists a unique domination $\psi'\colon Y\to Z$ such that $\psi'\circ\phi = \psi$.
    \end{enumerate}
    We call such $\phi$ a \emph{quotient domination}. The existence and uniqueness up to isomorphism of the quotient geometric data set are established in \cite[\Cref{1:def:Quotient}]{gds1}. We denote by $X/G$ the quotient geometric data set of $X$ by $G$.
\end{definition}

\begin{definition}[Observable distance, concentrate, {\cite[Definition 2.4]{hanika2022gds}}]
    Let $X$ and $Y$ be two geometric data sets. We define the \emph{observable distance} $\dconc(X,Y)$ between $X$ and $Y$ as the infimum of the quantity
    $\hausp{\kf^\lambda}(F_X\circ\phi,F_Y\circ\psi)$,
    where $\phi$ and $\psi$ run over all parameters of $\mu_X$ and $\mu_Y$, respectively; here, a \emph{parameter} of a Borel probability measure $\mu$ on a metric space $Z$ is a Borel measurable map $\phi\colon[0,1]\to Z$ such that $\phi_*\lambda = \mu$, where $\lambda$ denotes the Lebesgue measure on $[0,1]$. We say that a sequence $\{X_n\}_{n=1}^\infty$ \emph{concentrates} to $X$ if $\dconc(X_n,X) \to 0$ as $n\to\infty$.
\end{definition}

For geometric data sets $X,Y$ and a coupling $\pi\in\Pi(\mu_X,\mu_Y)$, we denote
\begin{equation*}
    \dconc^\pi(X,Y) \coloneqq \hausp{\kf^\pi}(F_X\circ\pr_1,F_Y\circ\pr_2).
\end{equation*}

\begin{proposition}[{\cite[\Cref{1:thm:dconcMinimum}]{gds1}}]\label{prop:dconcIsCouplingMin}
    For any geometric data sets $X$ and $Y$,
    \begin{equation*}
        \dconc(X,Y) = \min\{\dconc^\pi(X,Y) \mid \pi\in\Pi(\mu_X,\mu_Y)\}.
    \end{equation*}
\end{proposition}

\begin{theorem}[{\cite[\Cref{1:theorem:DconcPreservesDomination}]{gds1}}]\label{theorem:DconcPreservesDomination}
    Let $X,Y,X_n,Y_n$ be geometric data sets, $n = 1,2,\ldots$. If $X_n\preceq Y_n$ for $n=1,2,\ldots$, and if $X_n$ and $Y_n$ converge to $X$ and $Y$, respectively, as $n\to\infty$, then we have $X\preceq Y$.
\end{theorem}

\begin{definition}[Observable diameter, {\cite[\Cref{1:def:obsdiam}]{gds1}}]\label{def:obsdiam}
    Let $\mu$ be a Borel probability measure on $\R$ and $\kappa \in (0,1)$. The \emph{$(1-\kappa)$-partial diameter of $\mu$} is defined by
    \begin{equation*}
        \pd(\mu; 1-\kappa) \coloneqq \inf\left\{\diam I \;\middle|\; I \subset \R \text{ Borel measurable},\ \mu(I) \geq 1-\kappa\right\}.
    \end{equation*}
    For a geometric data set $X$, the \emph{$\kappa$-observable diameter of $X$} is defined as
    \begin{equation*}
        \od(X; -\kappa) \coloneqq \sup\left\{\pd(f_*\mu_X;\, 1-\kappa) \;\middle|\; f \in F_X\right\}.
    \end{equation*}
\end{definition}

\begin{lemma}[{\cite[Lemma~3.1(1)]{ozawa2015limit}}]\label{lem:pdRightContinuous}
    For any Borel probability measure $\mu$ on $\R$, the function $\kappa \mapsto \pd(\mu; 1-\kappa)$ is right-continuous.
\end{lemma}

\begin{lemma}[{\cite[Lemma~3.8]{ozawa2015limit}}]\label{lem:pdProkhorov}
    Let $\mu$ and $\nu$ be Borel probability measures on $\R$, and $\epsilon > 0$. If $\dpr(\mu,\nu) < \epsilon$, then for any $\kappa \in (0,1)$,
    \begin{equation*}
        \pd(\mu;\, 1-(\kappa+\epsilon)) \leq \pd(\nu;\, 1-\kappa) + 2\epsilon.
    \end{equation*}
\end{lemma}

\begin{lemma}\label{lem:OdiamRightContinuous}
    For any geometric data set $X$, the function $\kappa \mapsto \od(X; -\kappa)$ is right-continuous.
\end{lemma}
\begin{proof}
    For each $f \in F_X$, \Cref{lem:pdRightContinuous} gives $\pd(f_*\mu_X; 1-(\kappa+1/n)) \to \pd(f_*\mu_X; 1-\kappa)$ as $n\to\infty$. Taking the supremum over $f \in F_X$ yields $\od(X;-(\kappa+1/n)) \to \od(X;-\kappa)$.
\end{proof}

\begin{proposition}\label{prop:odMonotone}
    Let $X, Y$ be geometric data sets. If $X\preceq Y$, we have
    \begin{equation*}
        \od(X; -\kappa) \leq \od(Y; -\kappa)
    \end{equation*}
    for any $\kappa\in(0,1)$.
\end{proposition}
\begin{proof}
    Let $\phi\colon Y\to X$ be a domination, so that $\phi_*\mu_Y = \mu_X$ and $F_X\circ\phi\subset\overline{F_Y}$.
    For any $f\in F_X$ and $\epsilon > 0$, since $f\circ\phi\in\overline{F_Y}$, there exists $g\in F_Y$ with $\kf^Y(g, f\circ\phi) < \epsilon$.
    For any Borel set $I\subset\R$ with $g_*\mu_Y(I) \geq 1-\kappa$, we have
    \begin{equation*}
        \mu_Y\bigl((f\circ\phi)^{-1}(U(I,\epsilon))\bigr)
        \geq \mu_Y\bigl(g^{-1}(I)\bigr) - \mu_Y\bigl(\{y\mid |f(\phi(y))-g(y)|>\epsilon\}\bigr)
        \geq (1-\kappa) - \epsilon,
    \end{equation*}
    so $\pd((f\circ\phi)_*\mu_Y;\, 1-\kappa-\epsilon) \leq \diam(U(I,\epsilon)) \leq \diam(I) + 2\epsilon$.
    Taking the infimum over such $I$,
    \begin{equation*}
        \pd(f_*\mu_X;\, 1-\kappa-\epsilon) \leq \pd(g_*\mu_Y;\, 1-\kappa) + 2\epsilon \leq \od(Y;-\kappa) + 2\epsilon.
    \end{equation*}
    Letting $\epsilon\to 0$ and using \Cref{lem:pdRightContinuous}, we obtain $\pd(f_*\mu_X;\,1-\kappa)\leq\od(Y;-\kappa)$. Taking the supremum over $f\in F_X$ gives the result.
\end{proof}

\begin{proposition}[{\cite[\Cref{1:prop:odContinuity}]{gds1}}]\label{prop:odContinuity}
    Let $X, Y$ be geometric data sets. For any $\kappa \in (0,1)$ and $\delta > \dconc(X,Y)$, we have
    \begin{equation*}
        \od(X;\, -(\kappa+\delta)) \leq \od(Y;\, -\kappa) + 2\delta.
    \end{equation*}
    In particular, $\od$ is continuous with respect to $\dconc$, and hence also with respect to $\Box$.
\end{proposition}

\subsection{Box Distance}

\begin{theorem}[{\cite[\Cref{1:thm:BoxIsMin}]{gds1}}]\label{theorem:BoxIsMin}
    For geometric data sets $X,Y$, we have\begin{equation*}
        \Box(X,Y) = \min\left\{\Box^S_\pi(\overline{F_X},\overline{F_Y})\mmid\pi\in\Pi(\mu_X,\mu_Y),S\in\F(X\times Y)\right\},
    \end{equation*}
    where
    \begin{equation*}
        \Box^S_\pi(\overline{F_X},\overline{F_Y}) = \max\left\{1-\pi(S), 2\hauspdinf{S}(\overline{F_X}\circ\pr_1,\overline{F_Y}\circ\pr_2)\right\}.
    \end{equation*}
\end{theorem}

\begin{lemma}[{\cite[\Cref{1:lem:embeddedBox}]{gds1}}]\label{lemma:embeddedBox}
    Let $X$ and $Y$ be two geometric data sets and $(Z,\nu)$ a Borel probability measure space. Let $\phi\colon Z\to X$ and $\psi\colon Z\to Y$ be two Borel measurable maps such that $\phi_*\nu = \mu_X$ and $\psi_*\nu = \mu_Y$. For any Borel subset $S\subset Z$, we have
    \begin{equation*}
        \Box(X,Y) \leq \max\left\{ 1-\nu(S), 2\hauspdinf{S}(F_X\circ\phi,F_Y\circ\psi)\right\}.
    \end{equation*}
\end{lemma}

\begin{lemma}[{\cite[\Cref{1:lem:MinBoxMap}]{gds1}}]\label{lemma:MinBoxMap}
    Let $X$ be a metric space and let $S_n$, $S$ be closed subsets of $X$, $n = 1,2,\ldots$~. We assume that $S_n$ converges to $S$ in the weak Hausdorff sense. Then, for any subsets $F,G\subset\lip1(X)$, there exists a map $u\colon \overline{F}\to\overline{G}$ such that
    \begin{equation*}
        \dinf{S}(f,u(f)) \leq \liminf_{n\to\infty} \hauspdinf{S_n}(F,G)
    \end{equation*}
    for any $f\in \overline{F}$.
\end{lemma}

\begin{lemma}\label{lemma:ADominatedGDSPreserveBox}
    Let $\bar{X}$, $Y$, and $\bar{Y}$ be geometric data sets. Assume $Y\preceq \bar{Y}$. Then, there exists a geometric data set $X$ such that $X\preceq \bar{X}$ and $\Box(X,Y)\leq \Box\!\left(\bar{X}, \bar{Y}\right)$.
\end{lemma}
\begin{proof}
    By \Cref{theorem:BoxIsMin}, there exist a coupling $\pi\in\Pi(\mu_{\bar{X}},\mu_{\bar{Y}})$ and a closed set $S\subset \bar{X}\times \bar{Y}$ such that
    \begin{equation*}
        \max\!\left\{1-\pi(S), 2\hauspdinf{S}(\overline{F_{\bar{X}}}\circ\pr_1,\overline{F_{\bar{Y}}}\circ\pr_2)\right\} = \Box\!\left(\bar{X}, \bar{Y}\right).
    \end{equation*}
    Take a domination $\phi\colon\bar{Y}\to Y$. By \Cref{lemma:MinBoxMap}, there exists a map $u\colon \overline{F_{\bar{Y}}}\to\overline{F_{\bar{X}}}$. We set $F \coloneqq u(F_Y\circ\phi)$, so that
    \begin{equation*}
        \hauspdinf{S}(F\circ\pr_1,F_Y\circ\phi\circ\pr_2) \leq \hauspdinf{S}(F_{\bar{X}}\circ\pr_1,F_{\bar{Y}}\circ\pr_2).
    \end{equation*}
    We set $X\coloneqq \bar{X}/F$. \Cref{lemma:embeddedBox} yields that
    \begin{align*}
        \Box(X,Y) & \leq \max\!\left\{1-\pi(S), 2\hauspdinf{S}(F\circ\pr_1,F_Y\circ\phi\circ\pr_2)\right\} \\
                  & \leq \Box\!\left(\bar{X}, \bar{Y}\right).
    \end{align*}
    This completes the proof.
\end{proof}

\begin{proposition}[{\cite[\Cref{1:prop:DConcLessThanBoxDistance}]{gds1}}]\label{prop:DConcLessThanBox}
    For any geometric data sets $X$ and $Y$,
    \begin{equation*}
        \dconc(X,Y) \leq \Box(X,Y).
    \end{equation*}
\end{proposition}

\section{$\L$-Compact Class}
\begin{definition}[Monoidal]
    A subfamily $\L\subset\lip1(\R)$ is said to be \emph{monoidal} if it is a monoid with the identity element $\id_\R$ and is closed under pointwise convergence on $\R$.
\end{definition}

\begin{definition}[$\T$,$\B$,$\TB$]
    Let $\T$ be the family of translation functions,
    \begin{equation*}
        \T\coloneqq\{\, x\mapsto x + c \mid c\in\R\,\}
    \end{equation*}
    and let $\B$ be the family of bounding functions,
    \begin{equation*}
        \B\coloneqq\{\, b_R \coloneqq x\mapsto -R\lor x\land R \mid R\in[0,+\infty]\,\},
    \end{equation*}
    where $\lor$ is the maximum and $\land$ is the minimum.
    We set
    \begin{equation*}
        \TB\coloneqq\{\, x\mapsto l\lor (x + c)\land u \mid c\in\R,\, l\in[-\infty,+\infty),u\in(-\infty,+\infty],\,l\leq u\,\},
    \end{equation*}
    which is the smallest monoidal family that contains both $\T$ and $\B$.
\end{definition}

Hereafter, we fix a monoidal subfamily $\L\subset\lip1(\R)$. For any function $f\colon X\to \R$ and any family $F$ of functions from $X$ to $\R$, we denote
\begin{equation*}
    \L\circ f \coloneqq \{\,p\ncirc f\mid p\in\L\,\} \text{ and } \L\circ F \coloneqq \{\,p\ncirc f \mid p\in\L \text{ and } f\in F\,\}.
\end{equation*}
For any geometric data set $X$ and $p \in \lip1(\R)$, we define
\begin{equation*}
    p \circ X \coloneqq (X, \lip1(X, d_X), \mu_X)/(p \circ F_X),
\end{equation*}
where $\lip1(X, d_X)$ serves as the ambient feature family containing $p \circ F_X$.

\begin{definition}[$(\epsilon,\L)$-Covering number]
    Let $X$ be a geometric data set, $F\subset\overline{F_X}$ a subfamily, and $\epsilon > 0$ a real number. We define the $(\epsilon,\L)$-\emph{covering number} of $F$ as
    \begin{equation*}
        \cov(F,\epsilon;\L) \coloneqq \inf\{\,\#\net\mid \net\subset F \text{ is finite and } F\subset U(\L\circ\net,\epsilon;\,\kf^X)\,\},
    \end{equation*}
    where $\kf^X$ is the Ky Fan metric on $(X, \mu_X)$.
\end{definition}
\begin{lemma}\label{lem:CoveringNumberIsBoundedDConc}
    Let $X$ and $Y$ be geometric data sets. For any real numbers $\epsilon > 0$ and $\delta\in(0,\epsilon/2)$, if $\dconc(X,Y) < \delta$, then
    \begin{equation*}
        \cov(\overline{F_X},\epsilon;\L) \leq \cov(\overline{F_Y},\epsilon-2\delta;\L).
    \end{equation*}
\end{lemma}
\begin{proof}
    We may assume $\cov(\overline{F_Y},\epsilon-2\delta;\L) < +\infty$.
    By the definition of $(\epsilon-2\delta,\L)$-covering number of $Y$, there exists a finite subset $\net_Y\subset \overline{F_Y}$ such that
    \begin{equation*}
        \#\net_Y = \cov(\overline{F_Y},\epsilon-2\delta;\L) \text{ and } \overline{F_Y}\subset U(\L\circ\net_Y,\epsilon-2\delta;\,\kf^Y).
    \end{equation*}
    Since $\dconc(X,Y) < \delta$, there exists a coupling $\pi\in\Pi(\mu_X,\mu_Y)$ such that
    \begin{equation*}
        \hausp{\kf^\pi}(\overline{F_X}\circ\pr_1,\overline{F_Y}\circ\pr_2) < \delta.
    \end{equation*}
    By this, there exists a finite subset $\net_X\subset \overline{F_X}$ such that
    \begin{equation*}
        \#\net_X = \#\net_Y \text{ and } \hausp{\kf^\pi}(\net_X\circ\pr_1,\net_Y\circ\pr_2) < \delta.
    \end{equation*}
    For any $f\in\overline{F_X}$, we have
    \begin{align*}
        \kf^X(f,\L\circ\net_X)
         & \leq \kf^\pi(f\ncirc\pr_1,\overline{F_Y}\ncirc\pr_2) + \hausp{\kf^\pi}(\overline{F_Y}\ncirc\pr_2,\L\circ\net_X\ncirc\pr_1)             \\
         & < \delta + \hausp{\kf^Y}(\overline{F_Y},\L\circ\net_Y) + \hausp{\kf^\pi}(\L\circ\net_Y\ncirc\pr_2,\L\circ\net_X\ncirc\pr_1) \\
         & < \delta + (\epsilon-2\delta) + \delta = \epsilon,
    \end{align*}
    which yields that $\cov(\overline{F_X},\epsilon;\L) \leq \#\net_X = \cov(\overline{F_Y},\epsilon-2\delta;\L)$. This completes the proof.
\end{proof}

\begin{definition}[$\L$-Closed, $\L$-Compact]
    Let $X$ be a geometric data set and $F\subset \overline{F_X}$ a subfamily. The subfamily $F$ is said to be \emph{$\L$-closed} if $\L\circ F\subset F$. It is \emph{$\L$-compact} if it is $\L$-closed and the $(\epsilon,\L)$-covering number of $F$ is finite for any $\epsilon > 0$.

    The geometric data set $X$ is \emph{$\L$-closed} if $\overline{F_X}$ is $\L$-closed. It is \emph{$\L$-compact} if $\overline{F_X}$ is $\L$-compact.
\end{definition}

\begin{remark}
    As an example, consider the one-point geometric data set $(\{*\}, F, \delta_*)$ where $F \coloneqq \{\, * \mapsto c \mid c \in \R \,\}$ and $\delta_*$ is the Dirac measure at $*$. This is $\T$-compact: $F$ is $\T$-closed, and taking $\net \coloneqq \{* \mapsto 0\}$ gives $\T \circ \net = F$, so $\cov(F, \epsilon; \T) = 1$ for any $\epsilon \geq 0$. On the other hand, it is not $\{\id_\R\}$-compact, since $F \cong \R$ is not totally bounded in $\kf^{\delta_*}$.
\end{remark}

\begin{proposition}\label{prop:LCptIsDconcClosed}
    Let $\{X_n\}_{n=1}^\infty$ be a sequence of $\L$-compact geometric data sets and let $X$ be a geometric data set. If $X_n$ concentrates to $X$, then $X$ is $\L$-compact.
\end{proposition}
\begin{proof}
    For each $n\in\{1,2,\ldots\}$, we set $\delta_n \coloneqq \dconc(X_n,X) + \frac{1}{n}$. There exists a coupling $\pi_n\in\Pi(\mu_X,\mu_{X_n})$ such that
    \begin{equation*}
        \hausp{\kf^{\pi_n}}(\overline{F_X}\circ\pr_1, \overline{F_{X_n}}\circ\pr_2) < \delta_n.
    \end{equation*}
    Take any $f\in \overline{F_X}$ and any $p\in\L$. There exists $f_n\in \overline{F_{X_n}}$ such that
    \begin{equation*}
        \kf^{\pi_n}(f\circ\pr_1, f_n\circ\pr_2) < \delta_n,
    \end{equation*}
    and there exists $g_n\in \overline{F_X}$ such that $\kf^{\pi_n}(g_n\circ\pr_1, p\circ f_n\circ\pr_2) < \delta_n$. Thus,
    \begin{align*}
        \kf^X(p\circ f, g_n) & \leq \kf^{\pi_n}(p\circ f\circ\pr_1, p\circ f_n\circ\pr_2) + \kf^{\pi_n}(p\circ f_n\circ\pr_2, g_n\circ\pr_1) \\
                             & < \kf^{\pi_n}(f\circ\pr_1, f_n\circ\pr_2) + \delta_n < 2\delta_n \to 0
    \end{align*}
    as $n\to\infty$. Thus $p\circ f\in\overline{F_X}$, and we conclude that $X$ is $\L$-closed.

    For any $\epsilon > 0$, there exists $n\in\{1,2,\ldots\}$ such that $3\dconc(X_n, X) < \epsilon$. By \Cref{lem:CoveringNumberIsBoundedDConc}, we have
    \begin{equation*}
        \cov(\overline{F_X},\epsilon;\L) \leq \cov\left(\overline{F_{X_n}},\frac{\epsilon}{3};\L\right) < +\infty,
    \end{equation*}
    so that $X$ is $\L$-compact. This completes the proof.
\end{proof}

\begin{definition}[$\L$-Compact class]
    We define the \emph{$\L$-compact class}, denoted by $\D/\L$, as the set of all isomorphism classes of $\L$-compact geometric data sets.
\end{definition}
\begin{remark}
    The well-definedness of $\D/\L$ follows from \Cref{prop:LCptIsDconcClosed}, which implies that the $\L$-compactness of $X$ does not depend on the choice of isomorphism. This proposition also proves that $\D/\L$ is $\dconc$-closed. Since $\dconc \leq \Box$ by \Cref{prop:DConcLessThanBox}, $\D/\L$ is $\Box$-closed, and hence $\Box$-complete since $(\D, \Box)$ is complete.
\end{remark}

\begin{definition}[$\epsilon$-covering number and $\epsilon$-capacity]
\label{def:CoveringNumberCapacity}
    Let $X$ be a metric space, $\net$ a finite subset of $X$, and $\epsilon > 0$. We say $\net$ is an \emph{$\epsilon$-net} if $X \subset B(\net, \epsilon)$. We say $\net$ is an \emph{$\epsilon$-discrete net} if $d_X(x, y) > \epsilon$ for any distinct $x, y \in \net$. Define
    \begin{align*}
        \cov(X; \epsilon) &\coloneqq \inf\{\#\net \mid \text{$\net$ is an $\epsilon$-net of $X$}\}, \\
        \capa(X; \epsilon) &\coloneqq \sup\{\#\net \mid \text{$\net$ is an $\epsilon$-discrete net of $X$}\}.
    \end{align*}
    We call $\cov(X; \epsilon)$ the \emph{$\epsilon$-covering number} and $\capa(X; \epsilon)$ the \emph{$\epsilon$-capacity} of $X$. Note that $\cov(X;\epsilon)$ is distinct from the $(\epsilon,\L)$-covering number $\cov(F,\epsilon;\L)$ defined earlier.
\end{definition}

Let $X$ be a geometric data set. For any subfamily $F\subset F_X$, we denote
\begin{equation*}
    F/\L \coloneqq \{\L\circ f\mid f\in F\}.
\end{equation*}

\begin{proposition}\label{prop:LCoveringNumberIsStandardCoveringNumber}
    Let $X$ be a geometric data set. For any real number $\epsilon > 0$ and subfamily $F\subset F_X$, we have
    \begin{equation*}
        \cov(F,\epsilon;\L) \leq \cov\left(F/\L,\epsilon;\hausp{\kf^X}\right).
    \end{equation*}
\end{proposition}
\begin{proof}
    Take any $f_1,\ldots,f_N\in F$ and assume
    \begin{equation*}
        F/\L\subset \bigcup_{n=1}^N U(\L\circ f_n, \epsilon; \hausp{\kf^X}).
    \end{equation*}
    For any $f\in F$, there exists $n\in\{1,\ldots,N\}$ such that $\hausp{\kf^X}(\L\circ f,\L\circ f_n) < \epsilon$. In particular, there exists $p\in\L$ such that $\kf^X(f, p\circ f_n) < \epsilon$, so $f\in U(\L\circ f_n, \epsilon; \kf^X)$. Since $f$ is arbitrary,
    \begin{equation*}
        F\subset \bigcup_{n=1}^N U(\L\circ f_n, \epsilon; \kf^X).
    \end{equation*}
    This completes the proof.
\end{proof}

The following proposition characterizes $\mathcal{L}$-compactness in terms of compactness properties of $\overline{F_X}$ under the assumption that $\L$ is \emph{self-compact}.

\begin{definition}[self-compact]
    Let $\gamma$ denote the standard Gaussian measure on $\R$.
    A monoidal subfamily $\L\subset\lip1(\R)$ is said to be \emph{self-compact} if $\L/\L \coloneqq \{\L\circ p \mid p\in\L\}$, viewed as a subset of $(\F(\lip1(\R)),\hausp{\kf^\gamma})$, is compact.
\end{definition}
\begin{remark}
    The notion of self-compactness does not depend on the particular choice of $\gamma$: any Borel probability measure on $\R$ with full support yields the same compact sets in $\lip1(\R)$.
\end{remark}
\begin{proposition}\label{prop:TSubsetImpliesSelfCompact}
    The monoidal subfamilies $\{\id_\R\}$ and $\B$ are self-compact. More generally, any monoidal subfamily $\L \supset \T$ is self-compact; in particular, $\T$, $\TB$, and $\lip1(\R)$ are self-compact.
\end{proposition}
\begin{proof}
    Let $\L \in \{\{\id_\R\}, \B\}$ or $\T \subset \L$. For any $p \in \L$, there exists an isomorphism $q \in \L$ such that $q \circ p(0) = 0$: if $\L \in \{\{\id_\R\}, \B\}$, then $p(0) = 0$ for all $p \in \L$, so we take $q \coloneqq \id_\R \in \L$; if $\T \subset \L$, take $q \coloneqq (\id_\R - p(0)) \in \T \subset \L$. Given a sequence $\{p_n\}_{n=1}^\infty \subset \L$, there exists a sequence $\{q_n\}_n \subset \L$ such that $(q_n \circ p_n)(0) = 0$ for each $n$. Since $q_n \circ p_n \in \lip1(\R)$ and $(q_n \circ p_n)(0) = 0$, by \Cref{lemma:BddIsCompact} there exist an extraction $\iota$ and $p \in \L$ such that $q_{\iota(n)} \circ p_{\iota(n)}$ converges to $p$ in $\kf^\gamma$. For any $\epsilon > 0$ and sufficiently large $n$, we obtain $\kf^\gamma(q_{\iota(n)} \circ p_{\iota(n)}, p) < \epsilon$. For any $k \in \L$, since $k \circ q_{\iota(n)}, k \circ q_{\iota(n)}^{-1} \in \L$, we have
    \begin{equation*}
        \kf^\gamma((k \circ q_{\iota(n)}) \circ p_{\iota(n)},\ k \circ p) < \epsilon
        \quad\text{and}\quad
        \kf^\gamma(k \circ p_{\iota(n)},\ (k \circ q_{\iota(n)}^{-1}) \circ p) < \epsilon,
    \end{equation*}
    and therefore $\hausp{\kf^\gamma}(\L \circ p_{\iota(n)}, \L \circ p) \leq \epsilon$. This completes the proof.
\end{proof}

\begin{remark}
    Define $\T_+ \coloneqq \{\, t \mapsto t + a \mid a \geq 0 \,\}$, the monoidal subfamily of $\lip1(\R)$ consisting of non-negative translations. Then $\T_+$ is \emph{not} self-compact. Indeed, set $F_n \coloneqq \T_+ \circ (\id_\R + n) = \{\, \id_\R + b \mid b \geq n \,\}$ for each $n \in \N$. For any $n < m$ and any $g = \id_\R + b \in F_m$ (so $b \geq m > n$), we have
    \begin{equation*}
        \kf^\gamma(\id_\R + n,\, g) = \min(b - n,\, 1) = 1,
    \end{equation*}
    since $|(\id_\R + n)(x) - g(x)| = b - n \geq 1$ is constant. Thus $\hausp{\kf^\gamma}(F_n, F_m) = 1$ for all $n \neq m$, so $\T_+/\T_+$ contains an infinite $(1)$-separated set and is not $\hausp{\kf^\gamma}$-compact.
\end{remark}

\begin{proposition}\label{prop:LCompactQuotientTotallyBounded}
    Assume $\L$ is self-compact. Let $X$ be a geometric data set. For any $\L$-compact subfamily $F\subset F_X$, the quotient $F/\L$ is $\hausp{\kf^X}$-totally bounded.
\end{proposition}
\begin{proof}
    Let $\epsilon > 0$. Since $F$ is $\L$-compact, there exist $f_1,\ldots,f_N\in F$ such that
    \begin{equation*}
        F\subset \bigcup_{n=1}^N U(\L\circ f_n, \epsilon; \kf^X).
    \end{equation*}
    There exists a compact set $K\subset X$ such that $\mu_X(K) > 1-\epsilon$. By \Cref{lem:dInftyUniformlyContinuous}, there exists $\delta > 0$ such that for all $n\in\{1,\ldots,N\}$ and $p,q\in\L$, if $\kf^\gamma(p,q) < \delta$, then $\dinf{f_n(K)}(p,q) < \epsilon$. Since $\L$ is self-compact, there exist $p_1,\ldots,p_M\in\L$ such that
    \begin{equation*}
        \L/\L\subset \bigcup_{m=1}^M U(\L\circ p_m, \delta; \hausp{\kf^\gamma}).
    \end{equation*}
    For any $f\in F$, there exist $i\in\{1,\ldots,N\}$ and $p\in\L$ such that $\kf^X(f,p\circ f_i) < \epsilon$. Thus, there exists $p_m\in\L$ such that
    \begin{equation*}
        \hausp{\kf^\gamma}(\L\circ p, \L\circ p_m) < \delta.
    \end{equation*}
    For any $k\in\L$, there exists $l\in\L$ such that $\kf^\gamma(k\circ p, l\circ p_m) < \delta$, so that
    \begin{align*}
        \kf^X(k\circ f, l\circ p_m\circ f_i)
        &\leq \kf^X(k\circ f,k\circ p\circ f_i) + \kf^X(k\circ p\circ f_i, l\circ p_m\circ f_i)\\
        &\leq \epsilon + \max\{1-\mu_X(K), \dinf{K}(k\circ p\circ f_i, l\circ p_m\circ f_i)\} \leq 2\epsilon.
    \end{align*}
    Similarly, there exists $j\in\L$ such that $\kf^\gamma(j\circ p, k\circ p_m) < \delta$, so that
    \begin{align*}
        \kf^X(j\circ f, k\circ p_m\circ f_i)
        &\leq \kf^X(j\circ f,j\circ p\circ f_i) + \kf^X(j\circ p\circ f_i, k\circ p_m\circ f_i)\\
        &\leq \epsilon + \max\{1-\mu_X(K), \dinf{K}(j\circ p\circ f_i, k\circ p_m\circ f_i)\} \leq 2\epsilon,
    \end{align*}
    so $\hausp{\kf^X}(\L\circ f,\, \L\circ(p_m\circ f_i)) \leq 2\epsilon$. Since $f\in F$ is arbitrary,
    \begin{equation*}
        F/\L\subset \bigcup_{n=1}^N \bigcup_{m=1}^M U(\L\circ(p_m\circ f_n), 2\epsilon; \hausp{\kf^X}).
    \end{equation*}
    This completes the proof.
\end{proof}

\begin{proposition}\label{prop:SelfCompactCharacterization}
    Assume $\L$ is self-compact. Let $X$ be an $\L$-closed geometric data set. The following \textup{(1)--(3)} are equivalent.
    \begin{enumerate}
        \item $X$ is $\L$-compact.
        \item $\overline{F_X}/\L$ is $\hausp{\kf^X}$-totally bounded.
        \item $\L\circ\F({\overline{F_X}}) \coloneqq \{\L\circ F\mid F\in\F(\overline{F_X})\}$ is $\hausp{\kf^X}$-compact.
    \end{enumerate}
\end{proposition}
\begin{proof}
    We prove $(3)\Rightarrow(2)\Leftrightarrow(1)$ and $(2)\Rightarrow(3)$.

    $(3)\Rightarrow(2)$: Since $\overline{F_X}/\L\subset\L\circ\F(\overline{F_X})$, this is immediate.

    $(2)\Leftrightarrow(1)$: The implication $(2)\Rightarrow(1)$ follows from \Cref{prop:LCoveringNumberIsStandardCoveringNumber}, and $(1)\Rightarrow(2)$ follows from \Cref{prop:LCompactQuotientTotallyBounded}.

    $(2)\Rightarrow(3)$: We prove total boundedness and completeness of $\L\circ\F(\overline{F_X})$ separately.

    Let us prove the $\hausp{\kf^X}$-total boundedness of $\L\circ\F(\overline{F_X})$. Let $\epsilon > 0$. By (2), there exist $\net=\{f_1,\ldots,f_N\}\subset\overline{F_X}$ such that
    \begin{equation*}
        \overline{F_X}/\L \subset \bigcup_{n=1}^N U(\L\circ f_n,\,\epsilon;\,\hausp{\kf^X}).
    \end{equation*}
    For any $F\in\L\circ\F(\overline{F_X})$, define
    \begin{equation*}
        \net_F \coloneqq \bigl\{f_n \bigm| \exists\, g\in F,\ \hausp{\kf^X}(\L\circ f_n,\L\circ g) < \epsilon\bigr\} \in 2^\net.
    \end{equation*}
    Since $\#\{\L\circ\net_{F'}\mid F'\in\L\circ\F(\overline{F_X})\} \leq 2^N$, it is sufficient to prove
    \begin{equation*}
        \hausp{\kf^X}(\L\circ\net_F, F)\leq\epsilon,
    \end{equation*}
    for all $F\in\L\circ\F(\overline{F_X})$. For any $p\in\L$ and $f_n\in\net_F$, there exists $g\in F$ such that
    \begin{equation*}
        \hausp{\kf^X}(\L\circ f_n, \L\circ g) < \epsilon,
    \end{equation*}
    so there exists $q\in\L$ such that $\kf^X(p\circ f_n, q\circ g) < \epsilon$. Since $q\circ g\in F$, we have $\L\circ\net_F\subset U(F,\epsilon;\kf^X)$. For the other inclusion, take any $g\in F$. There exists $f_n\in\net$ such that
    \begin{equation*}
        \hausp{\kf^X}(\L\circ f_n, \L\circ g) < \epsilon,
    \end{equation*}
    so that $f_n\in\net_F$. Since there exists $p\in\L$ such that $\kf^X(p\circ f_n,g)<\epsilon$, we have $F\subset U(\L\circ\net_F,\epsilon;\kf^X)$. Thus $\hausp{\kf^X}(\L\circ\net_F, F)\leq\epsilon$.

    We now show the $\hausp{\kf^X}$-completeness of $\L\circ\F(\overline{F_X})$. Note that
    \begin{equation*}
        \L\circ\F(\overline{F_X}) = \{F\in\F(\overline{F_X}) \mid \L\circ F = F\}
    \end{equation*}
    since $\L$ is monoidal. For any $\hausp{\kf^X}$-Cauchy sequence $\{F_n\}_{n=1}^\infty\subset\L\circ\F(\overline{F_X})$, by \Cref{lemma:FXbarBallIsCompact}, the $\kf^X$-completeness of $\overline{F_X}$ yields $F\in\F(\overline{F_X})$ such that $\{F_n\}_n$ converges to $F$. Take any $p\in\L$ and $f\in F$. By Hausdorff convergence $F_n\to F$, there exists a sequence $\{f_n\}_{n=1}^\infty$ with $f_n\in F_n$ converging to $f$ in $\kf^X$. Since each $F_n$ is $\L$-closed, we have $p\circ f_n\in F_n$. From \Cref{lemma:1LipPullbackReductionKyFanDistance}, $\{p\circ f_n\}_{n=1}^\infty$ converges to $p\circ f$, so $p\circ f\in F$ by Hausdorff convergence. Since $\L\circ F\subset F$ and $\id_\R\in\L$ implies $F\subset\L\circ F$, we have $\L\circ F = F$, and hence $F\in\L\circ\F(\overline{F_X})$.

    The total boundedness and completeness of $\L\circ\F(\overline{F_X})$ yield that it is compact. This completes the proof.
\end{proof}

\begin{proposition}\label{prop:LCptInheritedByDomination}
    Let $X$ be an $\L$-closed geometric data set and $Y$ an $\L$-compact geometric data set. If $X\preceq Y$ and $\L$ is self-compact, then $X$ is also $\L$-compact.
\end{proposition}
\begin{proof}
    Take a domination $\phi\colon Y\to X$ and define a pull-back map
    \begin{equation*}
        \phi^*\colon (\overline{F_X}/\L,\hausp{\kf^X}) \to (\overline{F_Y}/\L,\hausp{\kf^Y})
    \end{equation*}
    by
    \begin{equation*}
        \phi^*(\L\circ f) \coloneqq \L\circ f\circ \phi \text{ for } f\in \overline{F_X}.
    \end{equation*}
    This is well-defined: if $\L\circ f = \L\circ f'$, then for any $k\in\L$ there exists $l\in\L$ with $k\circ f = l\circ f'$, so $k\circ(f\circ\phi) = l\circ(f'\circ\phi)\in\L\circ(f'\circ\phi)$; by symmetry, $\L\circ(f\circ\phi) = \L\circ(f'\circ\phi)$.
    Since $\phi_*\mu_Y = \mu_X$, the map $\phi^*$ preserves the distances. The total boundedness of $\overline{F_Y}/\L$ yields that of $\overline{F_X}/\L$. By \Cref{prop:LCoveringNumberIsStandardCoveringNumber}, $\overline{F_X}$ is $\L$-compact.
\end{proof}

\begin{theorem}\label{theorem:LClosedIsLCompact}
    Assume $\L\supset\T$ and let $X$ be a geometric data set. Any $\L$-closed subfamily $F\subset F_X$ is $\L$-compact.
\end{theorem}
\begin{proof}
    Take a point $x\in X$ and any $\epsilon > 0$. For any $f\in F$, define translation functions $p_f(t)\coloneqq t-f(x)$ and $q_f(t)\coloneqq t+f(x)$. Since $p_f\in\T\subset\L$ and $F$ is $\L$-closed, we have $p_f\circ f\in\L\circ F\subset F$. Moreover, $(p_f\circ f)(x)=0$, so $p_f\circ f\in F\cap\bdd(X,x,0)$. Since the Ky Fan metric $\kf^X$ metrizes convergence in measure, \Cref{lemma:BddIsCompact} implies that $\bdd(X,x,0)$ is $\kf^X$-compact; in particular, $F\cap\bdd(X,x,0)$ is $\kf^X$-totally bounded. Hence there exists a finite subfamily $\net\subset F\cap\bdd(X,x,0)\subset F$ such that $F\cap\bdd(X,x,0)\subset U(\net,\epsilon;\,\kf^X)$. For any $f\in F$, since $p_f\circ f\in F\cap\bdd(X,x,0)$, there exists $g\in\net$ such that $\kf^X(p_f\circ f, g) < \epsilon$. By \Cref{lemma:1LipPullbackReductionKyFanDistance}, we have
    \begin{equation*}
        \kf^X(f,\,q_f\circ g) \leq \kf^X(p_f\circ f,\ p_f\circ q_f\circ g) = \kf^X(p_f\circ f,\, g) < \epsilon.
    \end{equation*}
    Since $q_f\in\T\subset\L$ and $\net\subset F$, we see that $f\in U(\L\circ\net,\epsilon;\, \kf^X)$. The arbitrariness of $f$ yields $F\subset U(\L\circ\net,\epsilon;\, \kf^X)$, so $\cov(F,\epsilon;\L)\leq\#\net<\infty$. Since $\epsilon$ is arbitrary, $F$ is $\L$-compact.
\end{proof}
\begin{corollary}[mm-spaces induce $\lip1(\R)$-compact geometric data sets]
    For any mm-space $X$, the geometric data set $\D(X) \coloneqq (X, \lip1(X,\R), \mu_X)$ induced by $X$ is $\lip1(\R)$-compact.
\end{corollary}
\begin{proof}
    The family $\lip1(X,\R)$ is $\lip1(\R)$-closed. Since $\lip1(\R)\supset\T$, \Cref{theorem:LClosedIsLCompact} implies the result.
\end{proof}
\section{$\Box$-Separability of the $\L$-Compact Class}
\begin{definition}[$\M(N,R)$]
    Let $N\geq 1$ be an integer and $R\geq 0$ a real number. We denote by $\M(N,R)$ the set of all probability measures on $[-R,R]^N$.
\end{definition}
\begin{lemma}[{\cite[Definition~5.37 and Lemma~1.17(3)]{shioya2016mmg}}]\label{lemma:NRMeasurementIsDPrCompact}
    For any integer $N\geq 1$ and any real number $R>0$, the set $\M(N,R)$ is $\dpr$-compact.
\end{lemma}

\begin{definition}[$\DM(N)$, $\DM(N,R)$, $\DM(X;N)$, $\DM(X;N,R)$]
    Let $N\geq 1$ be an integer. We set\begin{equation*}
        \DM(N) \coloneqq \{X\in\D\mid \#F_X \leq N\}.
    \end{equation*}
    For any real number $R\geq 0$, we define \begin{equation*}
        \DM(N,R) \coloneqq \{X\in\DM(N)\mid F_X\subset\lip1(X,[-R,R])\}.
    \end{equation*}
    Let $X$ be a geometric data set. The $(X,N)$-\emph{feature measurement} $\DM(X;N)$ and the $(X,N,R)$-\emph{feature measurement} $\DM(X;N,R)$ of $X$ are defined as
    \begin{equation*}
        \DM(X;N) \coloneqq \{Y\in\DM(N)\mid Y\preceq X\}
    \end{equation*}
    and
    \begin{equation*}
        \DM(X;N,R) \coloneqq \{b_R\circ Y\mid Y\in\DM(X;N)\}.
    \end{equation*}
    For a subfamily $\mathcal{E}\subset\D$, we define $\DM(\mathcal{E};N,R) \coloneqq \bigcup_{X\in\mathcal{E}}\DM(X;N,R)$.
\end{definition}

\begin{lemma}\label{lemma:1LipPullbackReductionDconc}
    Let $X$ and $Y$ be geometric data sets and $p \in \lip1(\R)$. Then,
    \begin{equation*}
        \dconc(p \circ X, p \circ Y) \leq \dconc(X, Y).
    \end{equation*}
\end{lemma}
\begin{proof}
    For any $\epsilon > \dconc(X, Y)$, there exists a coupling $\pi \in \Pi(\mu_X, \mu_Y)$ such that
    \begin{equation*}
        \hausp{\kf^\pi}(F_X \circ \pr_1, F_Y \circ \pr_2) < \epsilon.
    \end{equation*}
    By \Cref{lemma:1LipPullbackReductionKyFanDistance}, we have
    \begin{equation*}
        \hausp{\kf^\pi}(p \circ F_X \circ \pr_1, p \circ F_Y \circ \pr_2) \leq \hausp{\kf^\pi}(F_X \circ \pr_1, F_Y \circ \pr_2) < \epsilon.
    \end{equation*}
    By the arbitrariness of $\epsilon > \dconc(X, Y)$, we obtain $\dconc(p \circ X, p \circ Y) \leq \dconc(X, Y)$.
\end{proof}

\begin{lemma}\label{lemma:NRFeatureMeasurementIsBoxCompact}
    For any integer $N\geq 1$ and any real number $R>0$, the $(N,R)$-feature measurement $\DM(N,R)$ is $\Box$-compact.
\end{lemma}
\begin{proof}
    Define a map $\phi\colon(\M(N,R),\dpr)\to(\DM(N,R),\Box)$ as
    \begin{equation*}
        \phi(\mu) \coloneqq (\supp\mu,\{\pr_n\mid n=1,\ldots,N\},\mu)\quad \text{for } \mu\in\M(N,R).
    \end{equation*}
    We prove that $\phi$ is surjective as follows. For any $X\in\DM(N,R)$, there exist finitely many functions $f_1,\ldots,f_N\in F_X$, not necessarily distinct, such that $F_X = \{f_1,\ldots,f_N\}$. Set $\mu\coloneqq (f_1,\ldots,f_N)_*\mu_X$. The map $\Phi\coloneqq(f_1,\ldots,f_N)\colon X\to\supp\mu$ is an isometry satisfying $\Phi_*\mu_X=\mu$ and $\pr_n\circ\Phi=f_n$, so $\phi(\mu)\cong X$. Thus $\phi$ is surjective.

    We next prove that $\phi$ is continuous. Take any $\mu, \nu\in\M(N,R)$. Here $d_{\R^N}$ denotes the $\ell^\infty$ distance on $\R^N$, that is, $d_{\R^N}(x,y)=\max_n|x_n-y_n|$. For any real number $\epsilon > \dpr(\mu,\nu)$, by Strassen's theorem~\cite[Theorem~17.11]{cohn2013measure}, there exist a closed set $S\subset [-R,R]^N\times [-R,R]^N$ and a coupling $\pi\in\Pi(\mu,\nu)$ satisfying
    \begin{equation*}
        \max\{1-\pi(S), \sup\{\,d_{\R^N}(x,y)\mid (x,y)\in S\,\}\} < \epsilon.
    \end{equation*}
    For any $(x,y)\in S$ and each $n$, we have $|\pr_n(x)-\pr_n(y)| \leq d_{\R^N}(x,y) < \epsilon$, so $\kf^\pi(\pr_n\circ\pr_1, \pr_n\circ\pr_2) \leq \epsilon$. Since $\pi(S) > 1-\epsilon$, we obtain $\Box(\phi(\mu),\phi(\nu)) \leq 2\epsilon$.
    The arbitrariness of $\epsilon > 0$ yields the continuity of $\phi$.

    By \Cref{lemma:NRMeasurementIsDPrCompact} and the continuity of $\phi$, we see that $\DM(N,R)$ is $\Box$-compact. This completes the proof.
\end{proof}

For a geometric data set $X = (X, F_X, \mu_X)$, we write $\L\circ X \coloneqq (X, \L\circ F_X, \mu_X)$.

\begin{lemma}\label{lemma:LComposeBoxNonExpansive}
    For any geometric data sets $X$ and $Y$,
    \begin{equation*}
        \Box(\L\circ X, \L\circ Y) \leq \Box(X, Y).
    \end{equation*}
    In particular, $\L\circ\DM(N,R)$ is $\Box$-compact for any integer $N\geq 1$ and real number $R>0$.
\end{lemma}
\begin{proof}
    The inequality follows by the same argument as in the proof of \Cref{lemma:1LipPullbackReductionDconc}, using \Cref{lemma:1LipPullbackReductionDinf}. The compactness of $\L\circ\DM(N,R)$ follows from \Cref{lemma:NRFeatureMeasurementIsBoxCompact} and the continuity of the map $X\mapsto\L\circ X$ (cf.\ \Cref{lemma:LComposeBoxNonExpansive}).
\end{proof}

\begin{definition}[$\DM(N,R;\L,\epsilon)$]
    Let $N\geq 1$ be an integer and $R,\epsilon>0$ real numbers. We define $\DM(N,R;\L,\epsilon)$ as the set of all $X\in\D/\L$ for which there exist a finite subset $\net\subset F_X$ and a closed subset $S\subset X$ satisfying:
    \begin{enumerate}[label=\textup{(\roman*)}]
        \item (size) $\#\net \leq N$;
        \item (measure) $\mu_X(S) > 1-\epsilon$;
        \item (approximation) $F_X\subset U(\L\circ\net,\epsilon;\dinf{S})$;
        \item (boundedness) $\sup\{|f(x)| \mid f\in \net,\, x\in S\} \leq R$.
    \end{enumerate}
\end{definition}

\begin{lemma}\label{lemma:AllNREpsilonFeatureMeasurementsIsBoxDense}
    For any $\L$-compact geometric data set $X$ and any real number $\epsilon>0$, there exist an integer $N\geq 1$ and a real number $R>0$ such that $X\in\DM(N,R;\L,\epsilon)$.
\end{lemma}
\begin{proof}
    From the separability of $X$, there exists a compact set $K \subset X$ such that $\mu_X(K) > 1-\epsilon$.
    \Cref{lem:dInftyUniformlyContinuous} implies that there exists $\delta > 0$ such that $\dinf{K}(f,g) < \epsilon$ for all $f,g\in F_X$ with $\kf^X(f,g) < \delta$. From the $\L$-compactness of $X$, there exists a finite subset $\net\subset F_X$ such that
    \begin{equation*}
        F_X\subset U(\L\circ\net,\delta;\, \kf^X)\subset U(\L\circ\net,\epsilon;\, \dinf{K}).
    \end{equation*}
    Setting $N\coloneqq\#\net$ and $R$ to be the maximum of $|f(x)|$ over all $f\in\net$ and $x\in K$, we see that $X\in\DM(N,R;\L,\epsilon)$. This completes the proof.
\end{proof}

\begin{lemma}\label{lemma:NREpsilonFeatureNearbyNRMeasurement}
    Let $N\geq 1$ be an integer and let $R>0$ and $\epsilon > 0$ be real numbers. Then,
    \begin{equation*}
        \DM(N,R;\L,\epsilon) \subset U(\L\circ\DM(N,R),2\epsilon;\Box),
    \end{equation*}
    where
    \begin{equation*}
        \L\circ\DM(N,R) \coloneqq \{\L\circ Y\mid Y\in\DM(N,R)\}.
    \end{equation*}
\end{lemma}
\begin{proof}
    Take any $X \in \DM(N,R;\L,\epsilon)$. There exist a finite subset $\net\subset F_X$ and a closed subset $S$ such that
    \begin{gather*}
        \#\net \leq N,\ \mu_X(S) > 1-\epsilon,\ F_X\subset U(\L\circ\net,\epsilon;\dinf{S}),\\
        \text{ and } \sup\{|f(x)| \mid f\in \net, x\in S\} \leq R.
    \end{gather*}
    We define
    \begin{equation*}
        Y \coloneqq (X,\lip1(X,d_X),\mu_X)/(b_{R}\circ\net)\in\DM(N,R).
    \end{equation*}
    By \Cref{lemma:1LipPullbackReductionDinf}, it follows that
    \begin{align*}
        \hauspdinf{S}(F_X,\L\circ b_R\circ\net) & \leq \hauspdinf{S}(F_X,\L\circ\net) + \hauspdinf{S}(\L\circ\net,\L\circ b_R\circ\net) \\
                                                & < \epsilon + \hauspdinf{S}(\net, b_R\circ\net) = \epsilon + 0 = \epsilon,
    \end{align*}
    where $\hauspdinf{S}(\net, b_R\circ\net) = 0$ since $b_R\circ f = f$ on $S$ for all $f\in\net$. \Cref{lemma:embeddedBox} yields that
    \begin{align*}
        \Box(X,\L\circ Y) & \leq \max\{1-\mu_X(S),\ 2\hauspdinf{S}(F_X,\L\circ b_R\circ\net)\} < 2\epsilon.
    \end{align*}
    This completes the proof.
\end{proof}

\begin{theorem}\label{thm:DLIsBoxCompleteSeparable}
    The $\L$-compact class $\D/\L$ is $\Box$-complete and separable.
\end{theorem}
\begin{proof}
    By \Cref{prop:LCptIsDconcClosed}, $\D/\L$ is $\dconc$-closed, in particular, $\Box$-closed. The completeness of $(\D/\L,\Box)$ follows from the completeness of $(\D,\Box)$. \Cref{lemma:AllNREpsilonFeatureMeasurementsIsBoxDense,lemma:NREpsilonFeatureNearbyNRMeasurement} (applied with $\epsilon/2$) imply that
    \begin{equation*}
        \D/\L \subset \bigcap_{\epsilon>0} \bigcup_{N=1}^\infty U\left(\L\circ\DM(N,N),\epsilon;\Box\right),
    \end{equation*}
    where $\L\circ\DM(N,N) \coloneqq \{\L\circ Y\mid Y\in\DM(N,N)\}$ (here we replace $(N,R)$ by $(M,M)$ with $M\geq\max\{N,R\}$). By \Cref{lemma:LComposeBoxNonExpansive}, each $\L\circ\DM(N,N)$ is $\Box$-compact, hence $\Box$-separable. The countable union $\bigcup_{N=1}^\infty \L\circ\DM(N,N)$ is thus separable, and since $\D/\L$ is contained in the $\Box$-closure of this union, $\D/\L$ is $\Box$-separable. This completes the proof.
\end{proof}
\begin{lemma}\label{lemma:BoxBallOfNREpsilonFeatureISNREpsilonFeature}
    Let $N\geq 1$ be an integer, and let $R$, $\epsilon$, and $\delta$ be positive real numbers. Then
    \begin{equation*}
        U(\DM(N,R;\L,\delta),\epsilon;\Box) \subset \DM(N,R+\epsilon;\L,\delta+2\epsilon).
    \end{equation*}
\end{lemma}
\begin{proof}
    Take any $X\in U(\DM(N,R;\L,\delta),\epsilon;\Box)$. There exist $Y\in\DM(N,R;\L,\delta)$, a coupling $\pi\in\Pi(\mu_X,\mu_Y)$, and a closed set $S\subset X\times Y$ such that
    \begin{equation*}
        \max\left\{1-\pi(S),\hauspdinf{S}(F_X\circ\pr_1,F_Y\circ\pr_2)\right\} < \epsilon.
    \end{equation*}
    By the definition of $\DM(N,R;\L,\delta)$, there exists a finite subset $\net_Y\subset F_Y$ and a closed subset $S_Y$ such that
    \begin{gather*}
        \#\net_Y \leq N,\ \mu_Y(S_Y) > 1-\delta,\ F_Y\subset U(\L\circ\net_Y,\delta;\dinf{S_Y}),\\
        \text{ and } \sup\{|g(y)| \mid g\in \net_Y, y\in S_Y\} \leq R.
    \end{gather*}
    By this, there exist a finite subset $\net_X\subset F_X$ and a closed subset $S$ such that
    \begin{equation*}
        \#\net_X \leq N\text{ and } \hauspdinf{S}(\net_X\circ\pr_1,\net_Y\circ\pr_2) < \epsilon.
    \end{equation*}
    Setting $S_X \coloneqq \overline{\pr_1(S\cap(X\times S_Y))}$, we see that
    \begin{align*}
        \mu_X(S_X)    & \geq \pi(S\cap(X\times S_Y)) > 1 - \delta - \epsilon,                                           \\
        F_X\circ\pr_1 & \subset U\!\left(F_Y\circ\pr_2,\epsilon;\hauspdinf{S\cap(X\times S_Y)}\right)                   \\
                      & \subset U\!\left(\L\circ\net_Y\circ\pr_2,\epsilon+\delta;\hauspdinf{S\cap(X\times S_Y)}\right)  \\
                      & \subset U\!\left(\L\circ\net_X\circ\pr_1,2\epsilon+\delta;\hauspdinf{S\cap(X\times S_Y)}\right) \\
        F_X           & \subset U\!\left(\L\circ\net_X,\delta+2\epsilon;\dinf{S_X}\right),
    \end{align*}
    where the last step uses $\dinf{S\cap(X\times S_Y)}(f\circ\pr_1,g\circ\pr_1) = \dinf{S_X}(f,g)$ for $f,g\in\lip1(X)$, by continuity and $S_X = \overline{\pr_1(S\cap(X\times S_Y))}$.
    For any $f\in\net_X$ and $x\in S_X$, we have
    \begin{align*}
        |f(x)| \leq \sup\{|g(y)| \mid g\in \net_Y, y\in S_Y\} + \epsilon \leq R + \epsilon,
    \end{align*}
    so that $X\in\DM(N,R+\epsilon;\L,\delta+2\epsilon)$. This completes the proof.
\end{proof}

\begin{lemma}\label{lemma:BoxPrecompactnessTest}
    Let $\E$ be a subset of $\D/\L$. The following \textup{(1)--(3)} are equivalent.
    \begin{enumerate}
        \item $\E$ is $\Box$-precompact.
        \item For any $\epsilon > 0$, there exist an integer $N\geq 1$ and a real number $R>0$ such that $\E\subset U(\L\circ\DM(N,R),\epsilon;\Box)$.
        \item For any $\epsilon > 0$, there exist an integer $N\geq 1$ and a real number $R>0$ such that $\E\subset\DM(N,R;\L,\epsilon)$.
    \end{enumerate}
\end{lemma}
\begin{proof}
    \Cref{lemma:NREpsilonFeatureNearbyNRMeasurement} implies that $(3)\Rightarrow (2)$. It is clear that $(2)\Rightarrow (1)$. Let us prove that $(1)\Rightarrow (3)$. Take any $\epsilon>0$. By the precompactness of $\E$, there exists a finite subset $\net\subset\E$ such that $\E\subset U(\net,\epsilon/4;\Box)$. Since $\net$ is finite and from \Cref{lemma:AllNREpsilonFeatureMeasurementsIsBoxDense}, there exist an integer $N\geq 1$ and a real number $R>0$ such that $\net\subset \DM(N,R;\L,\epsilon/2)$. \Cref{lemma:BoxBallOfNREpsilonFeatureISNREpsilonFeature} yields that
    \begin{align*}
        \E & \subset U\!\left(\net,\frac{\epsilon}{4};\Box\right)                                        \\
           & \subset U\!\left(\DM\!\left(N,R;\L,\frac{\epsilon}{2}\right),\frac{\epsilon}{4};\Box\right)
        \subset \DM\!\left(N,R+\frac{\epsilon}{4};\L,\epsilon\right).
    \end{align*}
    This completes the proof.
\end{proof}

\section{$\L$-Pyramids}

In this section, $\L$-pyramids are regarded as closed subsets of $(\D/\L, \Box)$. We say that a sequence of $\L$-pyramids \emph{converges weakly} if it converges in the weak Hausdorff sense with respect to $\Box$ (see \Cref{def:WeakHausdorffConvergence}). We clarify conditions under which the weak Hausdorff limit of $\L$-pyramids is again an $\L$-pyramid, a key step toward the compactification in \Cref{sec:compactification}.

\begin{definition}[$\L$-Pyramid] The subset $\P\subset\D/\L$ is called an $\L$-pyramid if the following conditions (1), (2), and (3) are satisfied.
    \begin{enumerate}
        \item For any $X\in\D/\L$ and $Y\in\P$, if $X\preceq Y$, then $X\in\P$.
        \item For any $X, Y\in\P$, there exists $Z\in\P$ such that $X\preceq Z$ and $Y\preceq Z$.
        \item $\P$ is nonempty and $\Box$-closed.
    \end{enumerate}
\end{definition}

\begin{definition}[Associated $\L$-pyramid]
    For any $\L$-compact geometric data set $X$, we define the \emph{$\L$-pyramid associated with $X$} as
    \begin{equation*}
        \P(X;\L) \coloneqq \{Y\in\D/\L\mid Y\preceq X\}.
    \end{equation*}
\end{definition}
\begin{remark}
    By \Cref{theorem:DconcPreservesDomination}, any associated $\L$-pyramid is an $\L$-pyramid.
\end{remark}

We write $\P_X$ for $\P(X;\L)$ when the monoidal family $\L$ is clear from context. We denote by $\Pi_\L$ the set of all $\L$-pyramids.

\begin{lemma}[$\L$-compact domination refinement]\label{lemma:LCompactDominationRefinement}
    Let $\bar{X}$, $Y$ and $\bar{Y}$ be $\L$-compact geometric data sets with $Y\preceq \bar{Y}$ and let $\L$ be self-compact. Then, there exists an $\L$-compact geometric data set $X$ such that $X\preceq \bar{X}$ and $\Box(X,Y) \leq \Box\!\left(\bar{X}, \bar{Y}\right)$.
\end{lemma}
\begin{proof}
    By \Cref{lemma:ADominatedGDSPreserveBox}, there exists a geometric data set $X'$ such that $X'\preceq \bar{X}$ and $\Box(X',Y) \leq \Box\!\left(\bar{X}, \bar{Y}\right)$. Set $X\coloneqq \L\circ X'$. Since $X$ is $\L$-closed and $X\preceq \bar{X}$ (via the composition of the domination $\bar{X}\to X'$ with the quotient domination $X'\to X$), \Cref{prop:LCptInheritedByDomination} implies that $X$ is $\L$-compact. From \Cref{lemma:LComposeBoxNonExpansive}, we see that
    \begin{equation*}
        \Box(X,Y) = \Box(\L\circ X',\L\circ Y) \leq \Box(X',Y) \leq \Box\!\left(\bar{X}, \bar{Y}\right),
    \end{equation*}
    where $\L\circ Y \simeq Y$ since $Y$ is $\L$-compact (hence $\L$-closed and monoidal). This completes the proof.
\end{proof}

\begin{lemma}\label{lemma:PrecompactGDSInUpperBounded}
    Let $X$, $Y$, and $\bar{Z}$ be $\L$-compact geometric data sets with $X,Y\preceq \bar{Z}$. Then, there exists an $\L$-compact geometric data set $Z$ such that $X,Y\preceq Z\preceq \bar{Z}$. Moreover, for any integer $N\geq 1$, any real number $R>0$, and any $\epsilon > 0$, if $X, Y \in \DM(N,R;\L,\epsilon)$, then $Z\in\DM(2N,R;\L,2\epsilon)$.
\end{lemma}
\begin{proof}
    Take dominations $\bar{\phi}\colon\bar{Z}\to X$ and $\bar{\psi}\colon\bar{Z}\to Y$. We define
    \begin{equation*}
        Z\coloneqq \bar{Z}/(F_X\circ\bar{\phi}\cup F_Y\circ\bar{\psi})
    \end{equation*}
    and take the quotient domination $\xi\colon\bar{Z}\to Z$. There exist dominations $\phi\colon Z\to X$ and $\psi\colon Z\to Y$ such that $\phi\circ\xi = \bar{\phi}$ and $\psi\circ\xi = \bar{\psi}$. Since $\overline{F_Z}\circ \xi = \overline{F_X}\circ\bar{\phi}\cup \overline{F_Y}\circ\bar{\psi}$, the surjectivity of $\xi$ gives $\overline{F_Z} = \overline{F_X}\circ\phi\cup \overline{F_Y}\circ\psi$. The $\L$-closedness of $\overline{F_Z}$ follows from that of $\overline{F_X}$ and $\overline{F_Y}$: for any $p\in\L$ and $f = g\circ\phi\in\overline{F_X}\circ\phi$, we have $p\circ f = (p\circ g)\circ\phi\in\overline{F_Z}$ (and similarly for $\overline{F_Y}\circ\psi$). Since $\phi_*\mu_Z = \mu_X$, we have $\kf^Z(g\circ\phi,\, p\circ h\circ\phi) = \kf^X(g,\, p\circ h)$ for any $g,h\in\overline{F_X}$ and $p\in\L$; hence $\cov(\overline{F_Z},\epsilon;\L)\leq\cov(\overline{F_X},\epsilon;\L)+\cov(\overline{F_Y},\epsilon;\L)<+\infty$. Thus $Z$ is $\L$-compact.

    Take any integer $N\geq 1$ and any real number $R>0$ and assume $X,Y\in\DM(N,R;\L,\epsilon)$. There exist a finite subset $\net_X\subset F_X$ and a closed set $S_X\subset X$ such that
    \begin{gather*}
        \#\net_X \leq N,\ \mu_X(S_X) > 1-\epsilon,\ F_X\subset U(\L\circ\net_X,\epsilon;\dinf{S_X}),\\
        \text{ and } \sup\{|f(x)| \mid f\in \net_X, x\in S_X\} \leq R,
    \end{gather*}
    and similarly, there exist a finite subset $\net_Y\subset F_Y$ and a closed set $S_Y\subset Y$ such that
    \begin{gather*}
        \#\net_Y \leq N,\ \mu_Y(S_Y) > 1-\epsilon,\ F_Y\subset U(\L\circ\net_Y,\epsilon;\dinf{S_Y}),\\
        \text{ and } \sup\{|g(y)| \mid g\in \net_Y, y\in S_Y\} \leq R.
    \end{gather*}
    Setting $\net \coloneqq \net_X\circ\phi\cup\net_Y\circ\psi$ and $S \coloneqq \phi^{-1}(S_X)\cap\psi^{-1}(S_Y)$, we see that
    \begin{align*}
        1-\mu_Z(S)  & = 1-\mu_Z(\phi^{-1}(S_X)\cap\psi^{-1}(S_Y))\\
                    & \leq 1-\mu_Z(\phi^{-1}(S_X)) + 1-\mu_Z(\psi^{-1}(S_Y))\\
                    & = 1-\mu_X(S_X) + 1-\mu_Y(S_Y) < 2\epsilon,\\
        F_Z         & \subset F_X\circ\phi\cup F_Y\circ\psi\\
                    & \subset U(\L\circ\net_X,\epsilon;\dinf{S_X})\cup U(\L\circ\net_Y,\epsilon;\dinf{S_Y})\\
                    & \subset U\!\left(\L\circ\net_X\circ\phi,\epsilon;\,\dinf{\phi^{-1}(S_X)}\right)\cup U\!\left(\L\circ\net_Y\circ\psi,\epsilon;\,\dinf{\psi^{-1}(S_Y)}\right) \\
                    & \subset U\!\left(\L\circ\net_X\circ\phi\cup\L\circ\net_Y\circ\psi,\epsilon;\,\dinf{S}\right)\\
                    & = U(\L\circ\net,\epsilon;\dinf{S}),
    \end{align*}
    and $|f(z)| \leq R$ for any $z\in S$ and $f\in\net$, so that $Z\in\DM(2N,R;\L,2\epsilon)$. This completes the proof.
\end{proof}

\begin{lemma}\label{lemma:PrecompactSequenceDominatedByGDSSequenceDominatingTwoSequence}
    Let $\{X_n\}_{n=1}^\infty$, $\{Y_n\}_{n=1}^\infty$, and $\left\{\bar{Z}_n\right\}_{n=1}^\infty$ be sequences of $\L$-compact geometric data sets, and let $X$ and $Y$ be $\L$-compact geometric data sets. Assume the following conditions \textup{(a)} and \textup{(b):}
    \begin{enumerate}[label=\textup{(\alph*)}]
        \item $X_n$ and $Y_n$ are dominated by $\bar{Z}_n$ for all $n=1,2,\ldots$~.
        \item Each of $X_n$ and $Y_n$ converges to $X$ and $Y$ in the $\Box$-sense, respectively, as $n\to\infty$.
    \end{enumerate}
    Then, there exist a sequence $\{Z_n\}_{n=1}^\infty$ of $\L$-compact geometric data sets satisfying the following \textup{(1)} and \textup{(2):}
    \begin{enumerate}
        \item $X_n,Y_n \preceq Z_n \preceq \bar{Z}_n$ for all $n=1,2,\ldots$~.
        \item There exists a subsequence $\{Z_{n(m)}\}_{m=1}^\infty$ such that it converges to $Z$ in the $\Box$-sense.
    \end{enumerate}
\end{lemma}
\begin{proof}
    For all $n=1,2,\ldots$, \Cref{lemma:PrecompactGDSInUpperBounded} implies that there exists an $\L$-compact geometric data set $Z_n$ such that $X_n,Y_n\preceq Z_n\preceq \bar{Z}_n$ and $Z_n\in\DM(2N,R;\L,2\epsilon)$ for any integer $N\geq 1$ and real number $R>0$ where $X_n$ and $Y_n$ are both in $\DM(N,R;\L,\epsilon)$. Since $\{X_n\}_n$ and $\{Y_n\}_n$ are both $\Box$-precompact and by \Cref{lemma:BoxPrecompactnessTest}, the sequence $\{Z_n\}_n$ is also $\Box$-precompact. By \Cref{thm:DLIsBoxCompleteSeparable}, $\{Z_n\}_n$ has a $\Box$-convergent subsequence. This completes the proof.
\end{proof}

\begin{theorem}\label{theorem:PyramidWeakLimitIsAnLpyramidOrAnEmptySet}
    Assume that $\L$ is self-compact. If a sequence of $\L$-pyramids converges weakly, then the weak limit is an $\L$-pyramid or an empty set.
\end{theorem}

\begin{proof}
    Let $\{\P_n\}_{n=1}^\infty$ be a sequence of $\L$-pyramids and $\P$ a closed subset of $\D/\L$ such that $\P_n$ converges weakly to $\P$ as $n\to\infty$. It suffices to prove that $\P$ satisfies the conditions (1) and (2) of the definition of $\L$-pyramid.

    Let us show (1). Take any $X\in\D/\L$ and $Y\in\P$ and assume $X\preceq Y$. The weak convergence implies that there exists a sequence $\{Y_n\}_{n=1}^\infty\subset\D/\L$ such that
    \begin{equation*}
        Y_n\in\P_n \text{ for all } n=1,2,\ldots, \text{ and } \Box(Y_n,Y) \to 0 \text{ as } n\to\infty.
    \end{equation*}
    For any $n=1,2,\ldots$, by \Cref{lemma:LCompactDominationRefinement}, there exists an $\L$-compact geometric data set $X_n$ such that $X_n\preceq Y_n$ and $\Box(X_n,X) \leq \Box(Y_n,Y)$. Since $X$ is the $\Box$-limit of the sequence $\{X_n\}_n$, we observe that $X\in\P$. This means (1).

    We check (2). Take any $X$ and $Y\in\P$. The weak convergence yields that there exist a sequence $\{X_n\}_{n=1}^\infty\subset\D/\L$ such that
    \begin{equation*}
        X_n\in\P_n \text{ for all } n=1,2,\ldots, \text{ and } \Box(X_n,X) \to 0 \text{ as } n\to\infty,
    \end{equation*}
    and a sequence $\{Y_n\}_{n=1}^\infty\subset\D/\L$ such that
    \begin{equation*}
        Y_n\in\P_n \text{ for all } n=1,2,\ldots, \text{ and } \Box(Y_n,Y) \to 0 \text{ as } n\to\infty.
    \end{equation*}
    For all $n=1,2,\ldots$, since $X_n$ and $Y_n$ are both in $\P_n$, there exists $\bar{Z}_n\in\P_n$ such that $X_n\preceq \bar{Z}_n$ and $Y_n\preceq \bar{Z}_n$. By \Cref{lemma:PrecompactSequenceDominatedByGDSSequenceDominatingTwoSequence}, there exist a sequence $\{Z_n\}_{n=1}^\infty$, a subsequence $\left\{Z_{n(m)}\right\}_{m=1}^\infty$, and an $\L$-compact geometric data set $Z$ such that
    \begin{equation*}
        Z_n\in\P_n \text{ for all } n=1,2,\ldots, \text{ and } \Box\!\left(Z_{n(m)},Z\right) \to 0 \text{ as } m\to\infty.
    \end{equation*}
    Since $Z$ is the $\Box$-limit of the subsequence $\{Z_{n(m)}\}_m$, we see that $Z\in\P$. \Cref{theorem:DconcPreservesDomination} shows that $X\preceq Z$ and $Y\preceq Z$. This means (2).

    Since $\P$ satisfies conditions (1) and (2) of the definition of $\L$-pyramid, $\P$ is an $\L$-pyramid or an empty set.
\end{proof}

\begin{theorem}\label{theorem:BoundedLPyramidWeakLimitIsAnLpyramid}
    Let $A$ be a compact subset of $\R$ and $p\in\lip1(\R,A)$. Assume $\L$ contains $p$ and is self-compact. If a sequence of $\L$-pyramids converges weakly, then the weak limit is an $\L$-pyramid.
\end{theorem}
\begin{proof}
    Let $\{\P_n\}_{n=1}^\infty$ be a sequence of $\L$-pyramids and $\P$ a closed subset of $\D/\L$ where $\P_n$ converges weakly to $\P$ as $n\to\infty$. By \Cref{theorem:PyramidWeakLimitIsAnLpyramidOrAnEmptySet}, $\P$ is an $\L$-pyramid or an empty set. It suffices to show that $\P$ is nonempty.

    Take a real number $R>0$ such that $A\subset[-R,R]$. For any $n\in\{1,2,\ldots\}$, since $\P_n$ is nonempty, there exist $X_n\in\P_n$ and $f_n\in F_{X_n}$. We define
    \begin{equation*}
        Y_n \coloneqq \L\circ\left(X_n/{\left\{p\circ f_n\right\}}\right)\in\L\circ\DM(1,R).
    \end{equation*}
    We have $Y_n = X_n/\L\circ \left\{p\circ f_n\right\}\in\P_n$. By \Cref{lemma:LComposeBoxNonExpansive}, $\L\circ\DM(1,R)$ is $\Box$-compact, so there exist a subsequence $\{Y_{n(m)}\}_{m=1}^\infty$ and $Y\in\L\circ\DM(1,R)$ such that $\Box(Y_{n(m)},Y)\to 0$. Since $Y_{n(m)}\in\P_{n(m)}$, weak Hausdorff convergence gives $Y\in\P$. This completes the proof.
\end{proof}

\begin{remark}\label{rem:TBNecessary}
    The monoidal families $\B$ and $\TB$ both satisfy the condition of \Cref{theorem:BoundedLPyramidWeakLimitIsAnLpyramid}. There exists a counterexample when $\L=\T$. Let $X_n\coloneqq \{0,n\}$, with its set of features $F_{X_n} \coloneqq \{p|_{X_n} \mid p \in \L\}$ and the normalized counting measure on $X_n$. Let $\P$ be a limit of the sequence $\{\P(X_n;\L)\}_{n=1}^\infty$.

    We show that $\P$ is empty. Since every element of $F_{X_n}$ has spread exactly $n$, any domination $\phi\colon X_n \to Y$ satisfying $F_Y \circ \phi \subset F_{X_n}$ must preserve this spread, so $\P(X_n;\L) = \{X_n\}$ for all $n$.

    Take any $X\in\P$. The sequence $\{X_n\}_n$ converges to $X$ in the $\Box$ sense as $n\to\infty$. By \Cref{prop:odContinuity} and $\dconc \leq \Box$, we have
    \begin{align*}
        +\infty & > \od\!\left(X;-\frac{1}{5}\right)                                      \\
                & \geq \lim_{n\to\infty} \od\!\left(X_n;-\frac{2}{5}\right) - \frac{2}{5} \\
                & = \lim_{n\to\infty} \diam X_n - \frac{2}{5} = +\infty,
    \end{align*}
    which is a contradiction. Hence, $\P$ is empty.
\end{remark}

\section{Staircase}
In this section, we define staircases and prove that the staircase map provides a topological embedding of $\L$-compact classes into a compact space. Our strategy is as follows: First, we introduce the notion of extractability for $\L$, which allows us to bound the observable diameter of a closed subset in terms of staircases and to establish a limit formula for the observable diameter. Second, we use the Ky Fan covering number of $F_X$ to control the convergence of $N$-measurements. Finally, combining these results, we establish an embedding theorem that realizes the concentration of geometric data sets via staircases.

\subsection{Definition and Basic Properties of Staircases}

\begin{definition}[Staircase]
    The sequence $\s=\{\s(N)\}_{N=1}^\infty\subset\F(\D)$ is called a \emph{staircase} if the following conditions (1), (2), and (3) are satisfied.
    \begin{enumerate}
        \item (measurement coherence) For any $\{N,M\}\subset\{1,2,\ldots\}$, if $M\leq N$, then $\s(M) = \overline{\DM(\s(N);M,M)}^\Box$.
        \item (directedness) For any $\{N,M\}\subset\{1,2,\ldots\}$, any $X\in\s(N)$, and any $Y\in\s(M)$, there exists $Z\in\s(N+M)$ such that $X\preceq b_N\circ Z$ and $Y\preceq b_M\circ Z$.
        \item (nonemptiness and closedness) Any $\s(N)$ is a nonempty $\Box$-closed subset in $\DM(N,N)$ for $N\geq 1$.
    \end{enumerate}
    We denote the set of all staircases by $\Sigma$ and define $d_\Sigma$ as
    \begin{equation*}
        d_\Sigma(\s,\su) \coloneqq \sum_{N=1}^\infty \frac{1}{2N\cdot 2^N}\cdot\haus{\Box}(\s(N),\su(N))\ \text{ for } \s,\su\in\Sigma.
    \end{equation*}
\end{definition}

\begin{definition}[Associated staircase]
    For any geometric data set $X$, we define the \emph{staircase $\s_X$ associated with $X$} as
    \begin{equation*}
        \s_X(N)\coloneqq \overline{\DM(X;N,N)}^\Box \text{ for all } N\geq 1.
    \end{equation*}
    For any $\L$-pyramid $\P$, we define the \emph{staircase $\s_\P$ associated with $\P$} as
    \begin{equation*}
        \s_\P(N)\coloneqq \overline{\DM(\P;N,N)}^\Box \text{ for all } N\geq 1.
    \end{equation*}
\end{definition}

\begin{lemma}\label{lemma:DMMeasurementLipschitz}
    For any $\Box$-closed subsets $A$ and $B$ of $\D$ and any positive integer $N$ and real number $R > 0$,
    \begin{equation*}
        \haus{\Box}(\DM(A;N,R),\DM(B;N,R))\leq \haus{\Box}(A,B).
    \end{equation*}
\end{lemma}
\begin{proof}
    Take $\epsilon > \haus{\Box}(A,B)$. For any $X\in\DM(A;N,R)$, there exist $X_A\in A$ and $\net=\{f_1,\ldots,f_N\}\subset F_{X_A}$ such that $X = b_R\circ X_A/\net$. Since $\epsilon > \haus{\Box}(A,B)$, there exists $Y_B\in B$ such that $\Box(X_A,Y_B) < \epsilon$. By \Cref{theorem:BoxIsMin,lemma:MinBoxMap}, there exist $\pi\in\Pi(\mu_{X_A},\mu_{Y_B})$ and $S\in\F(X_A\times Y_B)$ and $u\colon\overline{F_{X_A}}\to\overline{F_{Y_B}}$ such that
    \begin{equation*}
        \max\{1-\pi(S), 2\dinf{S}(f_n\circ\pr_1, u(f_n)\circ\pr_2)\} < \epsilon.
    \end{equation*}
    We write $Y = b_R\circ Y_B/u(\net)$ and take quotient dominations $\phi\colon (X_A,\lip1(X_A),\mu_{X_A})\to X$ and $\psi\colon (Y_B,\lip1(Y_B),\mu_{Y_B})\to Y$. Thus,
    \begin{align*}
        \Box(X,Y) &\leq \max\{1-\pi(S), 2\hauspdinf{S}(F_X\circ\phi\circ\pr_1, F_Y\circ\psi\circ\pr_2)\}\\
        &\leq \max\{1-\pi(S), 2\hauspdinf{S}(b_R\circ\net\circ\pr_1, b_R\circ u(\net)\circ\pr_2)\} < \epsilon.
    \end{align*}
    By the symmetry of $A$ and $B$, $\haus{\Box}(\DM(A;N,R),\DM(B;N,R)) < \epsilon$. This completes the proof.
\end{proof}

\begin{lemma}\label{lemma:DMPrecEq}
    Let $N,M$ be two natural number and let $R,S>0$ be real number with $N < M, R < S$. For any subset $\E\subset \D$, we obtain
    \begin{equation*}
        \overline{\DM(\overline{\DM(\E;M,S)}^\Box;N,R)}^\Box = \overline{\DM(\E;N,R)}^\Box
    \end{equation*}
\end{lemma}
\begin{proof}
    ($\subset$) Take any $X_{N,R}\in\DM(\overline{\DM(\E;M,S)}^\Box;N,R)$ and any real number $\epsilon > 0$. There exists $X_{\epsilon,M,S}\in\D$ and $X_{M,S}\in\DM(\E;M,S)$ such that
    \begin{equation*}
        \Box(X_{\epsilon,M,S},X_{M,S}) < \epsilon\, \text{and}\, X_{N,R}\in\DM(X_{\epsilon,M,S};N,R).
    \end{equation*}
    By \Cref{lemma:DMMeasurementLipschitz}, there exists $X_{\epsilon,N,R}\in\DM(X_{M,S};N,R)$ such that
    \begin{equation*}
        \Box(X_{N,R},X_{\epsilon,N,R}) < \epsilon.
    \end{equation*}
    There exists $X\in\E$ and $\net\subset F_X$ such that $X_{M,S} = b_S\circ X/\net$ and there exists $\{f_1,\ldots,f_N\}\subset F_{X_{M,S}}$ such that $X_{\epsilon,N,R} = b_R\circ X_{M,S}/\{f_1,\ldots,f_N\}$. Take quotient dominations $\phi\colon(X,\lip1(X))\to b_S\circ X/\net$ and $\psi\colon (b_S\circ X/\net, \lip1(b_S\circ X/\net))\to X_{\epsilon,N,R}$. Since
    \begin{align*}
        F_{X_{\epsilon,N,R}}\circ\psi\circ\phi = b_R\circ\{f_1,\ldots,f_N\}\circ\phi,
    \end{align*}
    there exists $g_1,\ldots,g_N\in\net$ such that $b_S\circ g_n = f_n\circ\phi$ for $n\in\{1,\ldots,N\}$. Since
    \begin{equation*}
        F_{X_{\epsilon,N,R}}\circ(\psi\circ\phi) = b_R\circ b_S\circ\{g_1,\ldots,g_N\} = b_R\circ\{g_1,\ldots,g_N\},
    \end{equation*}
    we prove $X_{\epsilon,N,R}\in\DM(\E;N,R)$. Thus,
    \begin{equation*}
        \Box(X_{N,R}, \DM(\E;N,R)) < \epsilon.
    \end{equation*}
    Since $\epsilon$ is arbitrary, $X_{N,R}\in\overline{\DM(\E;N,R)}^\Box$.

    ($\supset$) Take any $X_{N,R}\in\overline{\DM(\E;N,R)}^\Box$ and real number $\epsilon > 0$. There exists $X_{\epsilon,N,R}\in\DM(\E;N,R)$ such that
    \begin{equation*}
        \Box(X_{N,R},X_{\epsilon,N,R}) < \epsilon.
    \end{equation*}
    Since there exists $X\in\E$ and $\{f_1,\ldots,f_N\}$ such that
    \begin{equation*}
        X_{\epsilon,N,R} = b_R\circ X/\{f_1,\ldots,f_N\} = b_R\circ(b_S\circ X/\{f_1,\ldots,f_N\})/\{b_S\circ f_1,\ldots,b_S\circ f_N\},
    \end{equation*}
    we have
    \begin{align*}
        &\Box(X_{N,R}, \DM(\overline{\DM(\E;M,S)}^\Box;N,R))
        &\leq \Box(X_{N,R}, X_{\epsilon,N,R}) < \epsilon.
    \end{align*}
    Since $\epsilon$ is arbitrary, $X_{N,R}\in\overline{\DM(\overline{\DM(\E;M,S)}^\Box;N,R)}^\Box$.
\end{proof}

\begin{proposition}
    For any $\L$-pyramid $\P$, $\s_\P \in \Sigma$.
\end{proposition}
\begin{proof}
    Write $\s_\P(N) = \overline{\DM(\P;N,N)}^\Box$.
    \textit{Condition~(1).} By \Cref{lemma:DMPrecEq},
    \begin{equation*}
        \overline{\DM(\s_\P(N);M,M)}^\Box = \overline{\DM(\overline{\DM(\P;N,N)}^\Box;M,M)}^\Box = \overline{\DM(\P;M,M)}^\Box = \s_\P(M).
    \end{equation*}
    \textit{Condition~(2).} Take sequences $\{X_n\}_{n=1}^\infty\subset\DM(\P;N)$ and $\{Y_n\}_{n=1}^\infty\subset\DM(\P;M)$ with $b_N\circ X_n\to X\in\overline{\DM(\P;N,N)}^\Box$ and $b_M\circ Y_n\to Y\in\overline{\DM(\P;M,M)}^\Box$ as $n\to\infty$, respectively. By the pyramid axiom and a construction similar to \Cref{lemma:PrecompactGDSInUpperBounded}, there exists $Z_n\in\DM(\P;N+M)$ such that $X_n\preceq Z_n$ and $Y_n\preceq Z_n$, for $n\in\{1,2,\ldots\}$. Since $b_{N+M}\circ Z_n\in\overline{\DM(\P;N+M,N+M)}^\Box$, there exist an extraction $\iota$ and $Z\in\overline{\DM(\P;N+M,N+M)}^\Box$ such that $b_{N+M}\circ Z_{\iota(n)}\to Z$ as $n\to\infty$. Since $b_N\circ X_{\iota(n)} \preceq b_N\circ Z_{\iota(n)}$, we have $X \preceq b_N\circ Z$. In the same way, $Y \preceq b_M\circ Z$.
    \textit{Condition~(3).} Since $\P\neq\emptyset$, $\s_\P(N)=\overline{\DM(\P;N,N)}^\Box$ is nonempty. It is $\Box$-closed by definition and contained in $\DM(N,N)$ since $\DM(\P;N,N)\subset\DM(N,N)$.
\end{proof}

\begin{corollary}\label{cor:AssociatedStaircaseInSigma}
    For any $X \in \D/\L$, we have $\s_X \in \Sigma$.
\end{corollary}
\begin{proof}
    Since $\DM(\P_X;N,N) = \DM(X;N,N)$ for all $N \geq 1$, we have $\s_X = \s_{\P_X}$. The conclusion follows from the previous proposition.
\end{proof}

\begin{proposition}[Compactness of $\Sigma$]
    $\Sigma$ is compact.
\end{proposition}
\begin{proof}
    Take any sequence $\{\s_n\}_{n=1}^\infty$ of staircases. For any $N=1,2,\ldots$ and any extraction $\iota$, by \Cref{lem:weakHausdorffSubSequence} there exist an extraction $\iota'$ and a $\Box$-closed set $\s_\infty(N)\subset\DM(N,N)$ such that $\s_{\iota\circ\iota'(m)}(N)\to\s_\infty(N)$ in $\haus{\Box}$. By a diagonal argument, there exist an extraction $\iota$ and $\s_\infty=\{\s_\infty(N)\}_{N=1}^\infty$ such that $\s_{\iota(m)}(N)\to\s_\infty(N)$ in $\haus{\Box}$ for all $N=1,2,\ldots$~.

    Let us prove $\s_\infty\in\Sigma$. \textit{Condition~(1).} Take any $\{N,M\}\subset \{1,2,\ldots\}$ with $M\leq N$. Using condition~(1) for each $\s_{\iota(m)}$ (that is, $\overline{\DM(\s_{\iota(m)}(N);M,M)}^\Box=\s_{\iota(m)}(M)$) and \Cref{lemma:DMMeasurementLipschitz}, we observe that
    \begin{align*}
         & \haus{\Box}(\s_\infty(M),\overline{\DM(\s_\infty(N);M,M)}^\Box)                                                             \\
         & = \haus{\Box}(\s_\infty(M),\DM(\s_\infty(N);M,M))                                                                           \\
         & \leq \haus{\Box}(\s_\infty(M),\s_{\iota(m)}(M))                                                                                 \\
         & \quad\> + \haus{\Box}(\DM(\s_{\iota(m)}(N);M,M),\DM(\s_{\infty}(N);M,M))                                                       \\
         & \leq \haus{\Box}(\s_\infty(M),\s_{\iota(m)}(M)) + \haus{\Box}(\s_{\iota(m)}(N),\s_\infty(N)) \to 0
    \end{align*}
    as $m\to\infty$. This implies (1).

    \textit{Condition~(2).} Take $\{N,M\}\subset\{1,2,\ldots\}$, $X\in\s_{\infty}(N)$ and $Y\in\s_{\infty}(M)$. There exists a sequence $\{X_m\}_{m=1}^\infty$ such that
    \begin{equation*}
        X_m\in\s_{\iota(m)}(N) \text{ for all } m=1,2,\ldots, \text{ and } \Box(X_m,X) \to 0 \text{ as } m\to\infty,
    \end{equation*}
    and there exists a sequence $\{Y_m\}_{m=1}^\infty$ such that
    \begin{equation*}
        Y_m\in\s_{\iota(m)}(M) \text{ for all } m=1,2,\ldots, \text{ and } \Box(Y_m,Y) \to 0 \text{ as } m\to\infty.
    \end{equation*}
    For all $m=1,2,\ldots$, since $X_m\in\s_{\iota(m)}(N)$ and $Y_m\in\s_{\iota(m)}(M)$, there exists $Z_m\in\s_{\iota(m)}(N+M)$ such that $X_m\preceq b_N\circ Z_m$ and $Y_m\preceq b_M\circ Z_m$. By the $\Box$-compactness of $\DM(N+M,N+M)$, there exist an extraction $\iota'$ and $Z\in\DM(N+M,N+M)$ such that $Z_{\iota'(m)}\to Z$ in $\Box$. Since $Z_{\iota'(m)}\in\s_{\iota(\iota'(m))}(N+M)$ and $\s_{\iota(\iota'(m))}(N+M)\to\s_\infty(N+M)$, condition~(2) of \Cref{def:WeakHausdorffConvergence} gives $Z\in\s_\infty(N+M)$. By \Cref{theorem:DconcPreservesDomination}, $X\preceq b_N\circ Z$ and $Y\preceq b_M\circ Z$. This proves (2).

    \textit{Condition~(3).} It is clear that (3) holds. We prove the $d_\Sigma$-convergence. We see that
    \begin{align*}
         & \lim_{m\to\infty} d_\Sigma(\s_{\iota(m)},\s_\infty)                                                                                                  \\
         & \leq \lim_{m\to\infty} \sum_{N=1}^L \frac{1}{2N\cdot 2^N}\cdot\haus{\Box}(\s_{\iota(m)}(N),\s_\infty(N)) + \sum_{N=L+1}^\infty \frac{1}{2N\cdot 2^N} \\
         & \leq 2^{-L} \to 0 \text{ as } L\to\infty.
    \end{align*}
    This completes the proof.
\end{proof}

\begin{proposition}\label{prop:BoxBoundByDconcForMeasurement}
    Let $N$ and $M$ be positive integers. For any $X\in\DM(N)$ and $Y\in\DM(M)$, we have
    \begin{equation*}
        \Box(X,Y) \leq (N+M)\cdot\dconc(X,Y).
    \end{equation*}
\end{proposition}
\begin{proof}
    For any $\epsilon > \dconc(X, Y)$, there exists $\pi \in \Pi(\mu_X, \mu_Y)$ such that
    \begin{equation*}
        \hausp{\kf^\pi}(F_X \circ \pr_1, F_Y \circ \pr_2) < \epsilon.
    \end{equation*}
    Here, we have
    \begin{equation*}
        \{f_1, \ldots, f_N\} = F_X, \quad \{g_1, \ldots, g_M\} = F_Y
    \end{equation*}
    for some $f_1, \ldots, f_N \in F_X$ and $g_1, \ldots, g_M \in F_Y$. For any $n \in \{1, \ldots, N\}$ and $m \in \{1, \ldots, M\}$, there exist $g'_n \in F_Y$ and $f'_m \in F_X$ such that
    \begin{gather*}
        \kf^\pi(f_n \circ \pr_1, g'_n \circ \pr_2) < \epsilon, \\
        \kf^\pi(f'_m \circ \pr_1, g_m \circ \pr_2) < \epsilon.
    \end{gather*}
    Define
    \begin{align*}
        S &\coloneqq \bigcap_{n=1}^N \bigl\{(x,y) \in X \times Y \bigm| |f_n(x) - g'_n(y)| \leq \epsilon\bigr\} \\
        &\quad\cap \bigcap_{m=1}^M \bigl\{(x,y) \in X \times Y \bigm| |f'_m(x) - g_m(y)| \leq \epsilon\bigr\}.
    \end{align*}
    Then, $\pi(S) \geq 1 - (N+M) \cdot \epsilon$.

    For any $f \in F_X$, there exists $n \in \{1, \ldots, N\}$ such that $f = f_n$. Thus,
    \begin{align*}
        \dinf{S}(f \circ \pr_1, g'_n \circ \pr_2) &= \dinf{S}(f_n \circ \pr_1, g'_n \circ \pr_2) \leq \epsilon.
    \end{align*}
    Hence, $F_X \circ \pr_1 \subset U(F_Y \circ \pr_2, \epsilon; \dinf{S})$. By symmetry,
    \begin{equation*}
        \hauspdinf{S}(F_X \circ \pr_1, F_Y \circ \pr_2) \leq \epsilon.
    \end{equation*}
    Therefore,
    \begin{align*}
        \Box(X,Y) &\leq \max\{1 - \pi(S), 2\hauspdinf{S}(F_X \circ \pr_1, F_Y \circ \pr_2)\} \\
        &\leq \max\{(N+M) \cdot \epsilon, 2\epsilon\} \\
        &= (N+M) \cdot \epsilon.
    \end{align*}
    By the arbitrariness of $\epsilon > \dconc(X,Y)$, we obtain $\Box(X,Y) \leq (N+M) \cdot \dconc(X,Y)$.
\end{proof}

\begin{lemma}\label{lemma:DconcPiMap}
    Let $X$ and $Y$ be geometric data sets and $\pi\in\Pi(\mu_X,\mu_Y)$ a coupling. Then there exists a map $u\colon\overline{F_X}\to\overline{F_Y}$ such that
    \begin{equation*}
        \kf^\pi(f\circ\pr_1, u(f)\circ\pr_2) \leq \dconc^\pi(X,Y)
    \end{equation*}
    for any $f\in\overline{F_X}$.
\end{lemma}
\begin{proof}
    Since $\kf^\pi$ never exceeds $1$, the case $\dconc^\pi(X,Y) = 1$ is trivial. Assume $\dconc^\pi(X,Y) < 1$. Take any $f\in\overline{F_X}$. We find $g\in\overline{F_Y}$ satisfying $\kf^\pi(f\circ\pr_1, g\circ\pr_2) \leq \dconc^\pi(X,Y)$. Set $\epsilon \coloneqq (1-\dconc^\pi(X,Y))/3$. By the definition of $\overline{F_X}$, there exists a sequence $\{f_n\}_{n=1}^\infty$ in $F_X$ with $\kf^X(f_n,f) < \epsilon/n$. By the definition of $\dconc^\pi(X,Y)$, for each $n$ there exists $g_n\in F_Y$ such that
    \begin{equation*}
        \kf^\pi(f_n\circ\pr_1, g_n\circ\pr_2) < \dconc^\pi(X,Y) + \frac{\epsilon}{n}.
    \end{equation*}
    Thus $\kf^\pi(f\circ\pr_1, g_n\circ\pr_2) < \dconc^\pi(X,Y) + 2\epsilon/n$. Define
    \begin{equation*}
        S_n \coloneqq \left\{(x,y)\in X\times Y : |f(x)-g_n(y)| \leq \dconc^\pi(X,Y) + \frac{2\epsilon}{n}\right\}.
    \end{equation*}
    Then $\pi(S_n) > 1 - (\dconc^\pi(X,Y) + 2\epsilon/n) > 0$, so $S_n$ is nonempty and closed. By \Cref{lem:weakHausdorffSubSequence}, there exist an extraction $\iota_1$ and a closed set $S\subset X\times Y$ such that $\{S_{\iota_1(n)}\}_{n=1}^\infty$ converges to $S$ in the weak Hausdorff sense. By \Cref{lem:weakHausdorffMeasureBound}, $\pi(S) \geq 1-\dconc^\pi(X,Y) > 0$, so $S$ is nonempty. For any $(x,y)\in S$, there exists $(x_n,y_n)\in S_{\iota_1(n)}$ with $(x_n,y_n)\to(x,y)$, and thus
    \begin{align*}
        |f(x)-g_{\iota_1(n)}(y)| &\leq d_X(x,x_n) + |f(x_n)-g_{\iota_1(n)}(y_n)| + d_Y(y_n,y)\\
        &\leq d_X(x,x_n) + \dconc^\pi(X,Y) + \frac{2\epsilon}{n} + d_Y(y_n,y) \to \dconc^\pi(X,Y).
    \end{align*}
    Fix $(x_0,y_0)\in S$ and set $L \coloneqq \sup_{n \geq 1}|f(x_0)-g_{\iota_1(n)}(y_0)| < +\infty$. Then
    \begin{equation*}
        g_{\iota_1(n)} \in \bdd(Y, y_0, |f(x_0)| + L).
    \end{equation*}
    By \Cref{lemma:BddIsCompact}, there exist an extraction $\iota_2$ and $g\in\overline{F_Y}$ such that $\{g_{\iota_1\circ\iota_2(n)}\}_{n=1}^\infty$ converges pointwise to $g$. For any $(x,y)\in S$,
    \begin{equation*}
        |f(x) - g(y)| \leq \liminf_{n\to\infty} |f(x)-g_{\iota_1\circ\iota_2(n)}(y)| \leq \dconc^\pi(X,Y).
    \end{equation*}
    Hence $\kf^\pi(f\circ\pr_1, g\circ\pr_2) \leq \dconc^\pi(X,Y)$, and we set $u(f) \coloneqq g$.
\end{proof}

\begin{lemma}\label{lemma:DconcDominatedNearby}
    Let $X'$, $X$, and $Y$ be geometric data sets with $X' \preceq X$. Then, there exists a geometric data set $Y'$ such that $Y' \preceq Y$, $\#F_{Y'} \leq \#F_{X'}$, and $\dconc(X', Y') \leq \dconc(X, Y)$.
\end{lemma}
\begin{proof}
    Take a domination $\phi \colon X \to X'$. By \Cref{prop:dconcIsCouplingMin}, there exists a coupling $\pi \in \Pi(\mu_X, \mu_Y)$ such that $\dconc^\pi(X, Y) \leq \dconc(X, Y)$. By \Cref{lemma:DconcPiMap}, there exists a map $u \colon \overline{F_X} \to \overline{F_Y}$ such that
    \begin{equation*}
        \kf^\pi(f \circ \pr_1, u(f) \circ \pr_2) \leq \dconc^\pi(X, Y) \leq \dconc(X, Y)
    \end{equation*}
    for any $f \in \overline{F_X}$. Set $Y' \coloneqq Y / u(F_{X'} \circ \phi)$ and take the quotient domination $\psi \colon Y \to Y'$. Then $F_{Y'} \circ \psi = u(F_{X'} \circ \phi)$. Setting $\pi' \coloneqq (\phi \times \psi)_*\pi \in \Pi(\mu_{X'}, \mu_{Y'})$, we have $\kf^{\pi'}(f' \circ \pr_1, g' \circ \pr_2) = \kf^{\pi}(f' \circ \phi \circ \pr_1, g' \circ \psi \circ \pr_2)$ for any $f', g'$, so that
    \begin{align*}
        \dconc(X', Y') &\leq \hausp{\kf^{\pi'}}(F_{X'} \circ \pr_1, F_{Y'} \circ \pr_2) = \hausp{\kf^\pi}(F_{X'} \circ \phi \circ \pr_1, F_{Y'} \circ \psi \circ \pr_2) \\
        &= \hausp{\kf^\pi}(F_{X'} \circ \phi \circ \pr_1, u(F_{X'} \circ \phi) \circ \pr_2) \leq \dconc(X, Y).
    \end{align*}
    This completes the proof.
\end{proof}

\begin{proposition}\label{prop:StaircaseLipschitz}
    Let $X, Y \in \D/\L$. Then,
    \begin{equation*}
        d_\Sigma(\s_X, \s_Y) \leq \dconc(X,Y).
    \end{equation*}
\end{proposition}
\begin{proof}
    For any $X' \in \DM(X;N,N)$, \Cref{lemma:DconcDominatedNearby} implies that there exists $Y' \in \DM(Y;N,N)$ such that
    \begin{equation*}
        \dconc(X', Y') \leq \dconc(b_N\circ X, b_N\circ Y)\leq \dconc(X, Y).
    \end{equation*}
    By \Cref{prop:BoxBoundByDconcForMeasurement}, we have
    \begin{equation*}
        \Box(X', Y') \leq 2N \cdot \dconc(X', Y') \leq 2N \cdot \dconc(X, Y).
    \end{equation*}
    By symmetry, we obtain
    \begin{equation*}
        \haus{\Box}(\DM(X;N,N), \DM(Y;N,N)) \leq 2N \cdot \dconc(X, Y).
    \end{equation*}
    Therefore,
    \begin{align*}
        d_\Sigma(\s_X, \s_Y) &= \sum_{N=1}^\infty \frac{1}{2N \cdot 2^N} \cdot \haus{\Box}(\s_X(N), \s_Y(N)) \\
        &= \sum_{N=1}^\infty \frac{1}{2N \cdot 2^N} \cdot \haus{\Box}(\DM(X;N,N), \DM(Y;N,N)) \\
        &\leq \sum_{N=1}^\infty \frac{1}{2N \cdot 2^N} \cdot 2N \cdot \dconc(X,Y) \\
        &= \left(\sum_{N=1}^\infty \frac{1}{2^N}\right) \cdot \dconc(X,Y) = \dconc(X,Y).
    \end{align*}
    This completes the proof.
\end{proof}

\begin{proposition}\label{prop:StaircaseMeasurementEquivalence}
    Let $\P_n, \P$ be $\L$-pyramids for $n = 1, 2, \ldots$~. Then, the following conditions are equivalent:
    \begin{enumerate}[label=(\arabic*)]
        \item $\{\s_{\P_n}\}_{n=1}^\infty$ converges to $\s_\P$ in $(\Sigma, d_\Sigma)$.
        \item For any $N = 1, 2, \ldots$ and $R > 0$, $\haus{\Box}(\DM(\P_n; N, R), \DM(\P; N, R)) \to 0$ as $n \to \infty$.
    \end{enumerate}
\end{proposition}
\begin{proof}
    The implication $(2) \Rightarrow (1)$ is immediate from the definition of $d_\Sigma$. We show $(1) \Rightarrow (2)$. For any $N = 1, 2, \ldots$ and $R > 0$, take $M > \max\{N, R\}$. By \Cref{lemma:DMPrecEq}, $\overline{\DM(\s_{\P_n}(M); N, R)}^\Box = \overline{\DM(\P_n; N, R)}^\Box$ and $\overline{\DM(\s_\P(M); N, R)}^\Box = \overline{\DM(\P; N, R)}^\Box$. By \Cref{lemma:DMMeasurementLipschitz}, as $n \to \infty$,
    \begin{align*}
        \haus{\Box}(\DM(\P_n; N, R), \DM(\P; N, R))
        &= \haus{\Box}(\overline{\DM(\P_n; N, R)}^\Box, \overline{\DM(\P; N, R)}^\Box) \\
        &= \haus{\Box}(\overline{\DM(\s_{\P_n}(M); N, R)}^\Box, \overline{\DM(\s_{\P}(M); N, R)}^\Box) \\
        &= \haus{\Box}(\DM(\s_{\P_n}(M); N, R), \DM(\s_{\P}(M); N, R)) \\
        &\leq \haus{\Box}(\s_{\P_n}(M), \s_{\P}(M)) \to 0.
    \end{align*}
    This completes the proof.
\end{proof}

\subsection{Extractability and Observable Diameter}

\begin{definition}[Observable diameter of a set]
    For a subset $\E \subset \D$ and $\kappa \in (0,1)$, we define the \emph{$\kappa$-observable diameter} of $\E$ by
    \begin{equation*}
        \od(\E; -\kappa) \coloneqq \sup_{X \in \E} \od(X; -\kappa).
    \end{equation*}
\end{definition}

\begin{proposition}\label{prop:SetOdiamRightContinuous}
    For a subset $\E \subset \D$ and $\kappa \in (0,1)$, the function $\kappa \mapsto \od(\E; -\kappa)$ is right-continuous.
\end{proposition}
\begin{proof}
    By the monotonicity of $\od(\E; -\kappa)$ in $\kappa$, it suffices to show that
    \begin{equation*}
        \od(\E; -\kappa) \leq \limsup_{n \to \infty} \od\!\left(\E; -\left(\kappa + \frac{1}{n}\right)\right).
    \end{equation*}
    By \Cref{lem:OdiamRightContinuous}, $\od(X;-\kappa) = \lim_{n\to\infty}\od(X;-(\kappa+1/n))$ for each $X$. Thus,
    \begin{align*}
        \od(\E; -\kappa) &= \sup_{X \in \E} \od(X; -\kappa) \\
                         &= \sup_{X \in \E} \lim_{n \to \infty} \od\!\left(X; -\left(\kappa + \frac{1}{n}\right)\right) \\
                        &\leq \limsup_{n \to \infty} \sup_{X \in \E} \od\!\left(X; -\left(\kappa + \frac{1}{n}\right)\right) \\
                        &= \limsup_{n \to \infty} \od\!\left(\E; -\left(\kappa + \frac{1}{n}\right)\right).
    \end{align*}
    This completes the proof.
\end{proof}

Note that a limit formula for the observable diameter of a pyramid does not hold in general; extractability is introduced to simultaneously describe the relationship between weak convergence of pyramids and $\dconc$, and to provide such a formula. Technically, the extraction estimate $\phi(\kappa, r)$ gives a bound on $R$ such that any feature witnessing the observable diameter can be brought into an $(N, R)$-measurement without losing the observable diameter value.

\begin{definition}[Extractable, extraction estimate]
    We say that $\L$ is \emph{extractable on an open interval $(a,b) \subset (0,1)$} if there exists a function $\phi\colon (a,b) \times [0,+\infty) \to [0,+\infty)$ such that for any Borel probability measure $\mu$ on $\R$, any $\kappa \in (a,b)$, any $\epsilon \in (0, b-\kappa)$, and any $r > 0$, there exists $g \in \L$ satisfying
    \begin{equation*}
        \min\{r, \pd(\mu; 1-(\kappa + \epsilon))\} \leq \pd((b_{\phi(\kappa,r)})_* g_* \mu; 1-\kappa) + 2\epsilon.
    \end{equation*}
    Such a function $\phi$ is called an \emph{extraction estimate}.
\end{definition}

For a Borel probability measure $\mu$ on $\R$, the \emph{L\'evy mean} of $\mu$ is defined by
\begin{equation*}
    \lm(\mu) \coloneqq \inf\bigl\{ m \in \R \bigm| \mu((-\infty, m]) \geq 1/2 \bigr\}.
\end{equation*}

\begin{lemma}\label{lem:PartDiamTruncation}
    Let $\mu$ be a Borel probability measure on $\R$ with $\lm(\mu) = 0$, $\kappa \in (0, 1/2)$, and $R > 0$. If $\pd((b_R)_* \mu; 1 - \kappa) < R$, then $\pd(\mu; 1 - \kappa) = \pd((b_R)_* \mu; 1 - \kappa)$.
\end{lemma}
\begin{proof}
    Take any $\epsilon \in (0, R - \pd((b_R)_* \mu; 1 - \kappa))$. There exists a closed interval $I \subset \R$ such that
    \begin{equation*}
        (b_R)_* \mu(I) \geq 1 - \kappa, \quad \diam(I) < \pd((b_R)_* \mu; 1 - \kappa) + \epsilon < R.
    \end{equation*}
    Since $(b_R)_* \mu(I) > 1/2$ and $\lm((b_R)_* \mu) = 0$, we have $0 \in I$. Since $\diam(I) < R$, it follows that $I \subset (-R, R)$. Therefore, $\mu(I) = \mu(b_R^{-1}(I)) \geq 1 - \kappa$, and so
    \begin{equation*}
        \pd(\mu; 1 - \kappa) \leq \diam(I) < \pd((b_R)_* \mu; 1 - \kappa) + \epsilon.
    \end{equation*}
    By the arbitrariness of $\epsilon$, we obtain $\pd(\mu; 1 - \kappa) \leq \pd((b_R)_* \mu; 1 - \kappa)$. The reverse inequality $\pd((b_R)_*\mu; 1-\kappa) \leq \pd(\mu; 1-\kappa)$ holds since $b_R$ is $1$-Lipschitz: for any closed interval $I$ with $\mu(I) \geq 1-\kappa$, we have $b_R(I) \subset \R$ closed, $\diam(b_R(I)) \leq \diam(I)$, and $(b_R)_*\mu(b_R(I)) = \mu(b_R^{-1}(b_R(I))) \geq \mu(I) \geq 1-\kappa$.
\end{proof}

\begin{proposition}[{\cite[Lemma 4.1]{yokota2024obsdiam}}]\label{prop:Lip1Extractable}
    If $\T \subset \L$, then $\L$ is extractable on $(0, 1/2)$.
\end{proposition}
\begin{proof}
    Take any $\kappa \in (0, 1/2)$ and any $\epsilon > 0$ with $\kappa + \epsilon \in (0, 1/2)$.
    For any Borel probability measure $\mu$ on $\R$ and any $r \geq 0$,
    set $f \coloneqq \id_\R - \lm(\mu)$, so that $\lm(f_*\mu) = 0$.
    If $\pd((b_r)_* f_* \mu; 1 - (\kappa+\epsilon)) \geq r$, then
    \begin{align*}
        \min\{r, \pd(\mu; 1 - (\kappa + \epsilon))\} \leq r
        &\leq \pd((b_r)_* f_* \mu; 1 - (\kappa+\epsilon))\\
        &\leq \pd((b_r)_* f_* \mu; 1 - \kappa) + 2\epsilon.
    \end{align*}
    If $\pd((b_r)_* f_* \mu; 1 - (\kappa+\epsilon)) < r$,
    by \Cref{lem:PartDiamTruncation} applied to $f_*\mu$ with parameter $1-(\kappa+\epsilon)$,
    \begin{align*}
        \min\{r, \pd(\mu; 1 - (\kappa + \epsilon))\}
        &\leq \pd(f_* \mu; 1 - (\kappa + \epsilon)) \\
        &= \pd((b_r)_* f_* \mu; 1 - (\kappa + \epsilon)) \\
        &\leq \pd((b_r)_* f_* \mu; 1 - \kappa) + 2\epsilon.
    \end{align*}
    In both cases, taking $f \in \T \subset \L$ completes the proof.
\end{proof}

\begin{proposition}\label{prop:OdFromStaircase}
    Assume that $\L$ is extractable on an open interval $(a,b)$ with extraction estimate $\phi$. Let $\kappa \in (a,b)$ and $\epsilon \in (0, (b-\kappa)/2)$, so that $\kappa + 2\epsilon \in (a,b)$. Let $\E$ and $\G$ be closed subsets of $(\D/\L, \Box)$ and let $r \geq \od(\G; -\kappa) + 5\epsilon$. Suppose that there exists a positive integer
    \begin{equation*}
        N \geq \phi(\kappa + \epsilon, r) + \epsilon
    \end{equation*}
    such that
    \begin{equation*}
        \DM(\E;N,N) \subset U(\DM(\G;N,N), \epsilon; \Box).
    \end{equation*}
    Then, we have
    \begin{equation*}
        \od(\E; -(\kappa + 2\epsilon)) < \od(\G; -\kappa) + 5\epsilon.
    \end{equation*}
\end{proposition}
\begin{proof}
    By the definition of observable diameter, there exist $X \in \E$ and $f \in F_X$ such that, setting $\mu \coloneqq f_* \mu_X$, we have
    \begin{equation*}
        \od(\E; -(\kappa + 2\epsilon)) < \pd(\mu; 1 - (\kappa + 2\epsilon)) + \epsilon.
    \end{equation*}
    By the assumption of extractability, there exists $p \in \L$ such that
    \begin{equation*}
        \min\{r, \pd(\mu; 1 - (\kappa + 2\epsilon))\} \leq \pd((b_N)_* p_* \mu; 1 - (\kappa + \epsilon)) + 2\epsilon.
    \end{equation*}
    Since $X$ is $\L$-closed, $p \circ f \in \L \circ \overline{F_X} \subset \overline{F_X}$, so the quotient $X' \coloneqq b_N \circ (X / \{p \circ f\})$ is well-defined and is an element of $\DM(\E;N,N)$. By assumption, there exists $Y \in \DM(\G;N,N)$ such that $\Box(X', Y) < \epsilon$, and there exists $g \in F_Y$ such that
    \begin{equation*}
        \dpr((b_N)_* p_* f_* \mu_X, g_* \mu_Y) < \epsilon
    \end{equation*}
    by \cite[Lemma 1.26]{shioya2016mmg}.
    By \Cref{lem:pdProkhorov}, we have
    \begin{align*}
        \min\{r, \pd(\mu; 1 - (\kappa + 2\epsilon))\} &\leq \pd((b_N)_* p_* f_* \mu_X; 1 - (\kappa + \epsilon)) + 2\epsilon \\
        &\leq \pd(g_* \mu_Y; 1 - \kappa) + 4\epsilon \\
        &\leq \od(\G; -\kappa) + 4\epsilon < r.
    \end{align*}
    By the definition of $\min$, we obtain
    \begin{equation*}
        \pd(\mu; 1 - (\kappa + 2\epsilon)) \leq \od(\G; -\kappa) + 4\epsilon.
    \end{equation*}
    Finally, by the choice of $\mu$, we have
    \begin{align*}
        \od(\E; -(\kappa + 2\epsilon)) &< \pd(\mu; 1 - (\kappa + 2\epsilon)) + \epsilon \\
        &\leq \od(\G; -\kappa) + 5\epsilon.
    \end{align*}
    This completes the proof.
\end{proof}

\begin{proposition}\label{prop:OdiamLimit}
    Assume that $\L$ is extractable on an open interval $(a,b)$. Let $\P_n, \P$ be $\L$-pyramids for $n = 1, 2, \ldots$~. If $\{\s_{\P_n}\}_{n=1}^\infty$ converges to $\s_\P$, then for any $\kappa \in (a,b)$, we have
    \begin{align*}
        \od(\P; -\kappa) &= \lim_{\epsilon \to +0} \liminf_{n \to \infty} \od(\P_n; -(\kappa + \epsilon)) \\
        &= \lim_{\epsilon \to +0} \limsup_{n \to \infty} \od(\P_n; -(\kappa + \epsilon)).
    \end{align*}
\end{proposition}
\begin{proof}
    Since $\liminf_{n \to \infty} \od(\P_n; -(\kappa + \epsilon))$ and $\limsup_{n \to \infty} \od(\P_n; -(\kappa + \epsilon))$ are monotone non-increasing in $\epsilon > 0$, the right-hand sides exist in $[0, +\infty]$.

    First, we show that
    \begin{equation*}
        \lim_{\epsilon \to +0} \limsup_{n \to \infty} \od(\P_n; -(\kappa + \epsilon)) \leq \od(\P; -\kappa).
    \end{equation*}
    The inequality is trivial if $\od(\P; -\kappa) = +\infty$, so assume $\od(\P; -\kappa) < +\infty$. Let $\phi$ be the extraction estimate. For any $\epsilon \in (0, b - \kappa)$, set
    \begin{equation*}
        N \coloneqq \left\lceil \phi\!\left(\kappa + \frac{\epsilon}{2}, \od(\P; -\kappa) + \frac{5}{2}\epsilon\right) + \frac{\epsilon}{2} \right\rceil,
    \end{equation*}
    where $\lceil x \rceil$ denotes the smallest integer not less than $x$. By assumption, for sufficiently large $n$, we have
    \begin{equation*}
        \haus{\Box}(\s_{\P_n}(N), \s_\P(N)) \leq \frac{\epsilon}{2}.
    \end{equation*}
    By \Cref{prop:OdFromStaircase}, we have
    \begin{equation*}
        \od(\P_n; -(\kappa + \epsilon)) < \od(\P; -\kappa) + \frac{5}{2}\epsilon.
    \end{equation*}
    By the arbitrariness of $\epsilon$ and \Cref{prop:SetOdiamRightContinuous},
    \begin{align*}
        \lim_{\epsilon \to +0} \limsup_{n \to \infty} \od(\P_n; -(\kappa + \epsilon)) &\leq \lim_{\epsilon \to +0} \limsup_{n \to \infty} \left(\od(\P; -\kappa) + \frac{5}{2}\epsilon\right) \\
        &= \od(\P; -\kappa).
    \end{align*}

    Next, we show that
    \begin{equation*}
        \od(\P; -\kappa) \leq \lim_{\epsilon \to +0} \liminf_{n \to \infty} \od(\P_n; -(\kappa + \epsilon)).
    \end{equation*}
    The inequality is trivial if the right-hand side equals $+\infty$, so assume
    \begin{equation*}
        \lim_{\epsilon \to +0} \liminf_{n \to \infty} \od(\P_n; -(\kappa + \epsilon)) < +\infty.
    \end{equation*}
    For any sufficiently small $\epsilon > 0$, there exists an extraction $\iota$ such that
    \begin{align*}
        \lim_{n \to \infty} \od(\P_{\iota(n)}; -(\kappa + \epsilon)) &= \liminf_{n \to \infty} \od(\P_n; -(\kappa + \epsilon)) < +\infty,\\
        \od(\P_{\iota(n)}; -(\kappa + \epsilon)) &< +\infty,
    \end{align*}
    for $n\in\{1,2,\ldots\}$. We set
    \begin{equation*}
        R \coloneqq \sup_{n=1}^\infty \od(\P_{\iota(n)}; -(\kappa + \epsilon))\ \text{and}\ N \coloneqq \lceil \phi(\kappa + \epsilon, R + 5\epsilon) + \epsilon \rceil.
    \end{equation*}
    By \Cref{prop:OdFromStaircase},
    \begin{align*}
        \od(\P; -(\kappa + 3\epsilon))
        &\leq \lim_{n \to \infty} (\od(\P_{\iota(n)}; -(\kappa + \epsilon)) + 5\epsilon)\\
        &= \liminf_{n \to \infty} \od(\P_n; -(\kappa + \epsilon)) + 5\epsilon.
    \end{align*}
    By the arbitrariness of $\epsilon$ and \Cref{prop:SetOdiamRightContinuous},
    \begin{align*}
        \od(\P; -\kappa) &= \lim_{\epsilon \to +0} \od(\P; -(\kappa + 3\epsilon)) \\
        &\leq \lim_{\epsilon \to +0} \left(\liminf_{n \to \infty} \od(\P_n; -(\kappa + \epsilon)) + 5\epsilon\right) \\
        &= \lim_{\epsilon \to +0} \liminf_{n \to \infty} \od(\P_n; -(\kappa + \epsilon)).
    \end{align*}
    This completes the proof.
\end{proof}

\subsection{From Measurement Convergence to Concentration}\label{sec:covering}

In this subsection, we use covering number and capacity (recall \Cref{def:CoveringNumberCapacity}) to show that convergence of $N$-measurements implies concentration convergence.

\begin{lemma}
\label{lem:CapacityPreserved}
    Let $X, Y$ be $\L$-closed geometric data sets, $\epsilon > 0$, and $N > \capa(\overline{F_X}/\L; \epsilon)$. If
    \begin{equation*}
        \haus{\Box}(\DM(X; N), \DM(Y; N)) < \epsilon,
    \end{equation*}
    then $\capa(\overline{F_Y}/\L; 3\epsilon) < N$.
\end{lemma}
\begin{proof}
    Assume $\capa(\overline{F_Y}/\L; 3\epsilon) \geq N$. There exists a sequence $\{g_n\}_{n=1}^{N}\subset \overline{F_Y}$ such that for distinct $m,n$ in $\{1,\ldots,N\}$,
    \begin{equation*}
        \haus{\kf^Y}(\L\circ g_m, \L\circ g_n) > 3\epsilon.
    \end{equation*}
    Set $Y'\coloneqq Y/\{g_1, \ldots, g_{N}\}$ with quotient domination $\phi\colon Y\to Y'$; since $\phi_*\mu_Y=\mu_{Y'}$, each $g_n$ defines a feature of $Y'$ with the same Ky Fan distances. There exists $X' \in \DM(X; N)$ with $\dconc(X', Y') \leq \Box(X', Y') < \epsilon$. Thus there exists $\pi \in \Pi(\mu_{X'}, \mu_{Y'})$ with
    \begin{equation*}
        \haus{\kf^\pi}(F_{X'}\circ\pr_1, F_{Y'}\circ\pr_2) < \epsilon.
    \end{equation*}
    Since $\haus{\kf^\pi}(F_{X'}\circ\pr_1, F_{Y'}\circ\pr_2) < \epsilon$, for each $n\in\{1,\ldots,N\}$ there exists $f_n\in F_{X'}$ such that $\kf^\pi(f_n\circ\pr_1, g_n\circ\pr_2) < \epsilon$. Since $X'\preceq X$, there is a domination $\psi\colon X\to X'$ with $\psi_*\mu_X = \mu_{X'}$ and $F_{X'}\circ\psi \subset \overline{F_X}$. Set $h_m \coloneqq f_m\circ\psi \in \overline{F_X}$ for each $m$. Since $N > \capa(\overline{F_X}/\L; \epsilon)$, there exist $m,n\in\{1,\ldots,N\}$ such that $\hausp{\kf^X}(\L\circ h_m,\L\circ h_n) \leq \epsilon$. Since $\kf^{X'}(p\circ f_m, q\circ f_n) = \kf^X(p\circ h_m, q\circ h_n)$ for any $p,q\in\L$ (as $\psi_*\mu_X=\mu_{X'}$), we have $\hausp{\kf^{X'}}(\L\circ f_m,\L\circ f_n)\leq\epsilon$. Thus,
    \begin{align*}
        &\haus{\kf^{Y'}}(\L\circ g_m, \L\circ g_n)\\
        &\leq \haus{\kf^{\pi}}(\L\circ g_m\circ\pr_2, \L\circ f_m\circ\pr_1) + \epsilon + \haus{\kf^{\pi}}(\L\circ f_n\circ\pr_1, \L\circ g_n\circ\pr_2)\\
        &< 3\epsilon.
    \end{align*}
    This is a contradiction, so $\capa(\overline{F_Y}/\L; 3\epsilon) < N$.
\end{proof}

\begin{lemma}
\label{lem:CoveringNumberApproximation}
    Let $X$ be an $\L$-closed geometric data set and $N \geq \cov(\overline{F_X}/\L; \epsilon)$. Then
    \begin{equation*}
        \dconc(X, \L \circ \DM(X; N)) < \epsilon.
    \end{equation*}
\end{lemma}
\begin{proof}
    There exist $f_1, \ldots, f_N \in \overline{F_X}$ such that $\overline{F_X} \subset \L \circ \bigcup_{n=1}^N B(f_n, \epsilon; \kf^X)$. Set $X' \coloneqq X/\{f_1, \ldots, f_N\} \in \DM(X; N)$. Then
    \begin{equation*}
        \dconc(X, \L \circ X') \leq \hausp{\kf^X}(F_X, \L \circ \{f_1, \ldots, f_N\}) < \epsilon.
    \end{equation*}
    This completes the proof.
\end{proof}

\begin{lemma}[cf.\ {\cite[Lemma 3.2]{shioya2016mmg}}]
\label{lem:CovCapaRelation}
    For any metric space $X$ and $\epsilon > 0$, $\cov(X; \epsilon) \leq \capa(X; \epsilon)$.
\end{lemma}

\begin{lemma}\label{lem:LCptPyramidIsBoxCompact}
    Assume $\L$ is self-compact. For any $\L$-compact geometric data set $X$, the pyramid $\P_X$ is $\Box$-compact, hence $\dconc$-compact.
\end{lemma}
\begin{proof}
    Let $\{Y_n\}_{n=1}^\infty$ be any sequence in $\P_X$. For each $n$, take a domination $\phi_n\colon X\to Y_n$ and define
    \begin{equation*}
        G_n \coloneqq \overline{F_{Y_n}}\circ\phi_n \subset \overline{F_X}.
    \end{equation*}
    Each $G_n$ is $\L$-closed, so $G_n\in\L\circ\F(\overline{F_X})$. By the $\L$-compactness of $X$ and \Cref{prop:SelfCompactCharacterization}\,(3), the space $\L\circ\F(\overline{F_X})$ is $\hausp{\kf^X}$-compact. Hence there exist an extraction $\iota$ and $G\in\L\circ\F(\overline{F_X})$ such that $G_{\iota(n)}\to G$ in $\hausp{\kf^X}$.

    Set $Y\coloneqq X/G$. For any $\epsilon > 0$, let $K\subset X$ be compact with $\mu_X(K)>1-\epsilon$, and let $\delta>0$ be as in \Cref{lem:dInftyUniformlyContinuous}. For sufficiently large $n$, $\hausp{\kf^X}(G_{\iota(n)}, G) < \delta$, so $\hauspdinf{K}(G_{\iota(n)}, G) < \epsilon$. By \Cref{lemma:embeddedBox},
    \begin{equation*}
        \Box(Y_{\iota(n)}, Y) \leq \max\{1-\mu_X(K),\, 2\hauspdinf{K}(G_{\iota(n)}, G)\} < 2\epsilon.
    \end{equation*}
    Since $\epsilon$ is arbitrary, $Y_{\iota(n)}\to Y$ in $\Box$. Since $Y\in\P_X$ (as $Y\preceq X$), $\P_X$ is $\Box$-sequentially compact, hence $\Box$-compact. The $\dconc$-compactness follows from $\dconc\leq\Box$.
\end{proof}

\begin{proposition}[From $N$-measurements to concentration]
\label{prop:NMeasToConcentration}
    Assume $\L$ is self-compact. Let $X_n$ ($n = 1, 2, \ldots$) be $\L$-compact geometric data sets and $X$ an $\L$-compact geometric data set. If for every $N$, $\{\DM(X_n; N)\}_{n=1}^\infty$ converges to $\DM(X; N)$ in the Hausdorff $\Box$-topology, then $\{X_n\}_{n=1}^\infty$ concentrates to $X$.
\end{proposition}
\begin{proof}
    Take $\epsilon > 0$ and set $N \coloneqq \capa(\overline{F_X}/\L; \epsilon) + 1$. For large $n$,
    \begin{equation*}
        \haus{\Box}(\DM(X_n; N), \DM(X; N)) < \epsilon.
    \end{equation*}
    By \Cref{lem:CapacityPreserved,lem:CovCapaRelation}, $\cov(\overline{F_{X_n}}/\L; 3\epsilon), \cov(\overline{F_X}/\L; \epsilon) \leq N$. By \Cref{lem:CoveringNumberApproximation,lemma:LComposeBoxNonExpansive} and $\L\circ\DM(X; N)\subset\P_X$,
    \begin{align*}
        &\dconc(X_n, \P_X)\\
        &\leq \dconc(X_n, \L\circ\DM(X; N))\\
        &\leq \dconc(X_n, \L\circ\DM(X_n; N)) + \haus{\Box}(\L\circ\DM(X_n; N), \L\circ\DM(X; N))\\
        &\leq \dconc(X_n, \L\circ\DM(X_n; N)) + \haus{\Box}(\DM(X_n; N), \DM(X; N)) < 4\epsilon, \\
        &\dconc(X, \P_{X_n})\\
        &\leq \dconc(X, \L\circ\DM(X; N)) + \haus{\Box}(\DM(X; N), \DM(X_n; N)) < 2\epsilon.
    \end{align*}
    Thus $\max\{\dconc(X_n, \P_X), \dconc(X, \P_{X_n})\} \to 0$ as $n \to \infty$.

    Take any extraction $\iota_1$. By \Cref{lem:LCptPyramidIsBoxCompact}, $\P_X$ is $\dconc$-compact. Since $\dconc(X_{\iota_1(n)}, \P_X) \to 0$, there exists an extraction $\iota_2$ such that $X_{(\iota_1 \circ \iota_2)(n)}$ concentrates to some $Y \in \P_X$. Since $\dconc(X, \P_{X_{(\iota_1 \circ \iota_2)(n)}}) \to 0$, there exists $Z_n \in \P_{X_{(\iota_1 \circ \iota_2)(n)}}$ with $Z_n \to X$. We have $Z_n \preceq X_{(\iota_1 \circ \iota_2)(n)}$, $Z_n \to X$, and $X_{(\iota_1 \circ \iota_2)(n)} \to Y$. By \Cref{theorem:DconcPreservesDomination}, $X \preceq Y$. Since $Y \in \P_X$, we have $Y = X$. By the arbitrariness of $\iota_1$, $\{X_n\}$ concentrates to $X$.
\end{proof}

\subsection{Embedding Theorem}
\begin{lemma}\label{lemma:RemoveR}
    Assume $\T \subset \L$. Let $\E, \G \in \F(\D/\L, \Box)$, $N$ a positive integer, $R > 0$, and $\epsilon \in (0, 1/2)$. Suppose that
    \begin{gather*}
        \od(\E; -\epsilon) + \epsilon < R, \\
        \DM(\E; N, R) \subset U(\DM(\G; N, R), \epsilon; \Box).
    \end{gather*}
    Then $\DM(\E; N) \subset U(\DM(\G; N), 2(N+1)\epsilon; \Box)$.
\end{lemma}
\begin{proof}
    Take any $X\in\DM(\E;N)$. There exist $f_n\in\lip1(X), c_n\in\R$, for $n\in\{1,\ldots,N\}$, such that $F_X=\{f_n+c_n\}_{n=1}^N$ and $\lm((f_n)_*\mu_X) = 0$. Since $\T\subset\L$, we also have $X_0\coloneq(X,\{f_n\}_{n=1}^N,\mu_X)\in\DM(\E;N)$. By
    \begin{equation*}
        \pd\bigl((f_n)_*\mu_X;\, 1-\epsilon\bigr) \leq \od(\E; -\epsilon) < R - \epsilon\ \text{and}\ \lm((f_n)_*\mu_X) = 0,
    \end{equation*}
    there exists a closed interval $I_n \subset \R$ such that
    \begin{equation*}
        (f_n)_*\mu_X(I_n) \geq 1-\epsilon \quad \text{and} \quad I_n \subset [-(R-\epsilon), R-\epsilon].
    \end{equation*}
    We set
    \begin{equation*}
        X_R \coloneqq b_R\circ X_0\,\text{and}\,S_X\coloneqq \bigcap_{n=1}^N f_n^{-1}([-R,R])
    \end{equation*}
    and take quotient domination $\phi\colon(X_0,\lip1(X_0))\to b_R\circ X_0$, so that
    \begin{align*}
        \Box(X_0,X_R) &\leq \max\{1 - \mu_X(S_X),\ 2\hauspdinf{S_X}(F_{b_R\circ X_0}\circ\phi, F_{X_0})\}\\
        &\leq \max\{N\epsilon, 2\hauspdinf{S_X}(b_R\circ F_{X_0}, F_{X_0})\} = N\epsilon.
    \end{align*}

    By the assumption there exist $Y_0\in\DM(\G;N)$ such that $\Box(X_R,b_R\circ Y_0) < \epsilon$. By the definition of $\Box$, there exists $S_R\in\F(X_R\times b_R\circ Y_0)$ and $\pi_R\in\Pi(\mu_{X_R},\mu_{b_R\circ Y_0})$ such that
    \begin{equation*}
        1 - \pi_R(S_R) < \epsilon,\ 2\hauspdinf{S_R}(F_{X_R}\circ\pr_1, F_{b_R\circ Y_0}\circ\pr_2) < \epsilon.
    \end{equation*}
    Take the quotient domination $\psi\colon (Y_0,\lip1(Y_0))\to b_R\circ Y_0$. We set
    \begin{equation*}
        S_Y \coloneqq \psi^{-1}\left(\overline{\pr_2\left(S_R \cap \bigcap_{f\in F_{X_R}} f^{-1}([-(R-\epsilon), R-\epsilon])\times (b_R\circ Y_0)\right)}\right),
    \end{equation*}
    so that
    \begin{align*}
        1 - \mu_{Y_0}(S_Y) &= 1 - \psi_*\mu_{Y_0}\left(\overline{\pr_2\left(S_R \cap \bigcap_{f\in F_{X_R}} f^{-1}([-(R-\epsilon), R-\epsilon])\times (b_R\circ Y_0)\right)}\right)\\
        &\leq 1 - \pi_R\left(S_R \cap \bigcap_{f\in F_{X_R}} f^{-1}([-(R-\epsilon), R-\epsilon])\times (b_R\circ Y_0)\right)\\
        &\leq 1 - \pi_R(S_R) + \sum_{f\in F_{X_R}} (1 - \mu_{X_R}(f^{-1}([-(R-\epsilon), R-\epsilon])))\\
        &< \epsilon + \sum_{f\in F_{X_R}} (1 - f_*\phi_*\mu_{X_0}([-(R-\epsilon), R-\epsilon]))\\
        &= \epsilon + \sum_{f'\in F_{X_0}} (1 - (b_R)_*f'_*\mu_{X_0}([-(R-\epsilon), R-\epsilon])) < (N+1)\epsilon
    \end{align*}

    For all $g\in F_{Y_0}$ and $y\in S_Y$, there exists $(x,y')\in S_R$ and $g'\in F_{b_R\circ Y_0}$ such that
    \begin{equation*}
        \max_{f\in F_{X_R}} |f(x)| \leq R-\epsilon,\ d_{b_R\circ Y_0}(\psi(y),y') < \frac{\epsilon}{4},\ \text{and}\ b_R\circ g = g'\circ\psi.
    \end{equation*}
    We have
    \begin{equation*}
        |(b_R\circ g)(y)| = |(g'\circ\psi)(y)| < \max_{f\in F_{X_R}} |f(x)| + \frac{3\epsilon}{4} < R,
    \end{equation*}

    so that $(b_R\circ g)(y) = g(y)$, for any $y\in S_Y$. We prove
    \begin{align*}
        \Box(Y_0, b_R\circ Y_0) &\leq \max\{1-\mu_{Y_0}(S_Y), 2\hauspdinf{S_Y}(F_{b_R\circ Y_0}\circ\psi, F_{Y_0})\}\\
        &= \max\{1-\mu_{Y_0}(S_Y), 2\hauspdinf{S_Y}(b_R\circ F_{Y_0}, F_{Y_0})\}\\
        &\leq \max\{1-\mu_{Y_0}(S_Y), 2\max_{g\in F_{Y_0}} \dinf{S_Y}(b_R\circ g, g)\} \leq (N+1)\epsilon.
    \end{align*}

    By the triangle inequality, $\Box(X_0,Y_0) < 2(N+1)\epsilon$. There exist $S_0\in\F(X_0\times Y_0)$ and $\pi_0\in\Pi(\mu_{X_0},\mu_{Y_0})$ such that
    \begin{equation*}
        1 - \pi_0(S_0) < 2(N+1)\epsilon,\ 2\hauspdinf{S_0}(F_{X_0}\circ\pr_1,F_{Y_0}\circ\pr_2) < 2(N+1)\epsilon
    \end{equation*}
    By \Cref{lemma:ADominatedGDSPreserveBox}, we may replace $Y_0$ by a new element of $\DM(\G;N)$ and assume that there exist $g_1,\ldots, g_N\in F_{Y_0}$ such that $\dinf{S_0}(f_n\circ\pr_1, g_n\circ\pr_2) < (N+1)\epsilon$ for $n=1,2,\ldots,N$. We set $Y = (Y_0, \{g_n+c_n\}_{n=1}^N,\mu_{Y_0})$, so that
    \begin{align*}
        \Box(X,Y) &\leq \max\{1 - \pi_0(S_0),\ 2\hauspdinf{S_0}(F_X\circ\pr_1,F_Y\circ\pr_2)\}\\
        &= \max\{1 - \pi_0(S_0),\ 2\max_{n=1}^N \dinf{S_0}((f_n+c_n)\circ\pr_1,(g_n+c_n)\circ\pr_2)\}\\
        &= \max\{1 - \pi_0(S_0),\ 2\max_{n=1}^N \dinf{S_0}(f_n\circ\pr_1,g_n\circ\pr_2)\} < 2(N+1)\epsilon.
    \end{align*}
    Since $\T\subset\L$, $Y\in\DM(\G;N)$. This completes the proof.
\end{proof}

\begin{proposition}\label{prop:StaircaseToMeasurement}
    Assume $\T \subset \L$. Let $\P_n, \P$ be $\L$-pyramids for $n = 1, 2, \ldots$~. If $\od(\P; -\epsilon) < +\infty$ for all sufficiently small $\epsilon > 0$ and $\{\s_{\P_n}\}_{n=1}^\infty$ converges to $\s_\P$, then for any $N = 1, 2, \ldots$, $\{\DM(\P_n; N)\}_{n=1}^\infty$ converges to $\DM(\P; N)$ in $(\F(\D), \haus{\Box})$.
\end{proposition}
\begin{proof}
    Take any $\epsilon \in (0, 1/4)$. By \Cref{prop:OdiamLimit}, $\limsup_{n\to\infty} \od(\P_n; -2\epsilon) < +\infty$. Set
    \begin{equation*}
        R \coloneqq \max\!\left\{\limsup_{n\to\infty} \od(\P_n; -2\epsilon),\; \od(\P; -2\epsilon)\right\} + 3\epsilon < +\infty.
    \end{equation*}
    By \Cref{prop:StaircaseMeasurementEquivalence}, for sufficiently large $n$,
    \begin{equation*}
        \haus{\Box}(\DM(\P_n; N, R),\; \DM(\P; N, R)) < \epsilon, \od(\P_n; -2\epsilon) + 2\epsilon < R.
    \end{equation*}
    Applying \Cref{lemma:RemoveR} with $\epsilon$ replaced by $2\epsilon$ (and using the symmetry of $\haus{\Box}$ for the reverse inclusion),
    \begin{equation*}
        \haus{\Box}(\DM(\P_n; N),\; \DM(\P; N)) \leq 4(N+1)\epsilon.
    \end{equation*}
    Since $\epsilon$ is arbitrary, the conclusion follows.
\end{proof}

\begin{theorem}[Embedding into staircase space]
\label{thm:StaircaseEmbedding}
    If $\L$ contains $\T$, then $(\D/\L, \dconc)$ admits a topological embedding into $\Sigma$.
\end{theorem}
\begin{proof}
    Since $\T\subset\L$, the monoidal family $\L$ is self-compact by \Cref{prop:TSubsetImpliesSelfCompact}, which is required by \Cref{prop:NMeasToConcentration} below.
    Define $\Phi\colon \D/\L \to \Sigma$ by $\Phi(X) \coloneqq \s_X$. By \Cref{prop:StaircaseLipschitz}, $\Phi$ is $1$-Lipschitz with respect to $\dconc$.

    Take $X_n, X \in \D/\L$ and assume $\{\s_{X_n}\}$ converges to $\s_X$. Since $Y \preceq X$ implies $\od(Y; -\epsilon) \leq \od(X; -\epsilon)$ by \Cref{prop:odMonotone}, we have $\od(\P_X; -\epsilon) \leq \od(X; -\epsilon) < +\infty$ for any $\epsilon > 0$. By \Cref{prop:StaircaseToMeasurement}, $\{\DM(\P_{X_n}; N)\}$ converges to $\DM(\P_X; N)$ in the Hausdorff $\Box$-topology. Noting $\DM(\P_{X_n}; N) = \DM(X_n; N)$ and $\DM(\P_X; N) = \DM(X; N)$, by \Cref{prop:NMeasToConcentration}, $\{X_n\}$ concentrates to $X$. Thus $\Phi$ is a topological embedding.
\end{proof}

\section{Compactification by Pyramid}\label{sec:compactification}

In this section, we show that when $\L$ contains $\TB$, the space of $\L$-pyramids equipped with a suitable metric provides a compactification of $(\D/\L, \dconc)$.

\begin{lemma}\label{TBLCompact}
    Let $\P$ be an $\L$-pyramid, $N$ be natural number, and $R > 0$ be real number. If $\TB\subset\L$, $\DM(\P;N,R)$ is compact.
\end{lemma}
\begin{proof}
    For any $X\in\P$ and $\{f_1,\ldots,f_N\}\subset F_X$, since $\L\circ b_R\circ\{f_1,\ldots,f_N\}$ is $\L$-compact, $\L\circ b_R\circ X/\{f_1,\ldots,f_N\}\in\P$. Thus, $\L\circ\DM(\P;N,R)\subset \P$.

    For any $\{X_n\}_{n=1}^\infty\subset\DM(\P;N,R)$, since $\DM(N,R)$ is compact, there exist $X\in\DM(N,R)$ and an extraction $\iota$ such that $X_{\iota(n)}\to X$ as $n\to\infty$ in $\Box$. Since $\L\circ X_{\iota(n)}\in\P$, we have $\L\circ X\in\P$. Since $X\in\DM(\L\circ X;N,R)$, we prove that $\DM(\P;N,R)$ is compact.
\end{proof}

\begin{proposition}[Metrization of weak Hausdorff convergence]
\label{prop:WeakHausdorffMetrization}
    If $\TB \subset \L$, then for $\L$-pyramids $\P_n, \P$ ($n = 1, 2, \ldots$), the following are equivalent:
    \begin{enumerate}[label=(\arabic*)]
        \item $\P_n$ converges to $\P$ in the weak Hausdorff sense as $n \to \infty$.
        \item For any $N \in \mathbb{N}$ and $R > 0$, $\DM(\P_n; N, R)$ converges to $\DM(\P; N, R)$ in $\haus{\Box}$.
        \item For any $N \in \mathbb{N}$, $\DM(\P_n; N, N)$ converges to $\DM(\P; N, N)$ in $\haus{\Box}$.
    \end{enumerate}
\end{proposition}
\begin{proof}
    Since $\TB\subset\L$, we have $\L\circ\DM(\P;N,R)\subset\P$. We show $(1) \Rightarrow (2)$. For $n = 1, 2, \ldots$, define $f_n \in \lip1(\DM(\P; N, R), \Box)$ by $f_n(X) \coloneqq \Box(X, \DM(\P_n; N, R))$. For any $X\in\DM(\P; N, R)$, there exists $Y\in\P$ such that $X\in\DM(Y;N,R)$. By weakly Hausdorff convergence, there exists a sequence $Y_n\in\P_n$ converging to $Y$ in $\Box$. Then
    \begin{align*}
        f_n(X) = \Box(X, \DM(\P_n; N, R)) &\leq \Box(X, \DM(Y_n; N, R))\\
        &\leq \haus{\Box}(\DM(Y; N, R), \DM(Y_n; N, R))\\
        &\leq \Box(Y, Y_n) \to 0,
    \end{align*}
    where the last inequality follows from \Cref{lemma:DMMeasurementLipschitz}. We set $\epsilon_n\coloneqq \| f_n\|_\infty$. By the Arzel\`{a}--Ascoli theorem and \Cref{TBLCompact}, $\lim_{n \to \infty} \epsilon_n = 0$ and
   \begin{equation*}
    \DM(\P; N, R) \subset U(\DM(\P_n; N, R), \epsilon_n; \Box).
   \end{equation*}

    Suppose for contradiction that $\DM(\P_n; N, R) \not\to \DM(\P; N, R)$ in $\haus{\Box}$. Then there exist $\delta > 0$ and an extraction $\iota_1$ such that
    \begin{equation*}
        \DM(\P_{\iota_1(n)}; N, R) \not\subset U(\DM(\P; N, R), \delta; \Box).
    \end{equation*}
    Thus there exists $Y_n \in \DM(\P_{\iota_1(n)}; N, R)$ with $\Box(Y_n, \DM(\P; N, R)) \geq \delta$. Since $\DM(N, R)$ is $\Box$-compact, there exists an extraction $\iota_2$ such that $Y_{\iota_2(n)}$ converges to some $Y \in \DM(N, R)$ in $\Box$. By weak Hausdorff convergence and $\L\circ Y_n \in \P_n$, we have $\L\circ Y\in\P$, so $Y\in\DM(\L\circ Y; N, R)\subset\DM(\P; N, R)$. This is a contradiction. Thus $(1) \Rightarrow (2)$ holds.

    $(2) \Rightarrow (3)$ follows by taking $R = N$ in~(2).

    We show $(3) \Rightarrow (1)$. For any $X \in \P$ and $\epsilon > 0$, by \Cref{lemma:AllNREpsilonFeatureMeasurementsIsBoxDense,lemma:NREpsilonFeatureNearbyNRMeasurement}, there exists $N$ with $\Box(X, \L \circ \DM(X; N, N)) < \epsilon$. By $(3)$, for large $n$,
    \begin{equation*}
        \haus{\Box}(\DM(\P_n; N, N), \DM(\P; N, N)) < \epsilon.
    \end{equation*}
    Thus
    \begin{align*}
        &\Box(X, \P_n)\\
        &\leq \Box(X, \L \circ \DM(\P; N, N)) + \haus{\Box}(\L \circ \DM(\P_n; N, N), \L \circ \DM(\P; N, N)) < 2\epsilon,
    \end{align*}
    so $\lim_{n \to \infty} \Box(X, \P_n) = 0$.

    It remains to verify \Cref{def:WeakHausdorffConvergence}\,(2). Take any extraction $\iota$ and $Y_n \in \P_{\iota(n)}$ ($n = 1, 2, \ldots$) such that $\{Y_n\}_{n=1}^\infty$ $\Box$-converges to some $Y \in \D/\L$. We show $Y \in \P$. For any $\epsilon > 0$, by \Cref{lemma:AllNREpsilonFeatureMeasurementsIsBoxDense,lemma:NREpsilonFeatureNearbyNRMeasurement}, there exists $N$ such that $\Box(Y, \L \circ \DM(Y; N, N)) < \epsilon$. For sufficiently large $n$,
    \begin{equation*}
        \Box(Y, Y_n) < \frac{\epsilon}{2N}
        \quad \text{and} \quad
        \haus{\Box}(\DM(\P_{\iota(n)}; N, N),\, \DM(\P; N, N)) < \epsilon.
    \end{equation*}
    Since $Y_n \in \P_{\iota(n)}$, we have $\DM(Y_n; N, N) \subset \DM(\P_{\iota(n)}; N, N)$. By \Cref{lemma:DMMeasurementLipschitz,lemma:LComposeBoxNonExpansive},
    \begin{equation*}
        \haus{\Box}(\L \circ \DM(Y; N, N),\, \L \circ \DM(Y_n; N, N)) \leq \Box(Y, Y_n) < \epsilon.
    \end{equation*}
    Therefore,
    \begin{align*}
        \Box(Y, \P)
        &\leq \Box(Y,\, \L \circ \DM(\P; N, N)) \\
        &\leq \Box(Y,\, \L \circ \DM(\P_{\iota(n)}; N, N)) + \haus{\Box}(\L \circ \DM(\P_{\iota(n)}; N, N),\, \L \circ \DM(\P; N, N)) \\
        &\leq \Box(Y,\, \L \circ \DM(Y_n; N, N)) + \epsilon \\
        &\leq \Box(Y,\, \L \circ \DM(Y; N, N)) + \haus{\Box}(\L \circ \DM(Y; N, N),\, \L \circ \DM(Y_n; N, N)) + \epsilon \\
        &< 3\epsilon.
    \end{align*}
    By the arbitrariness of $\epsilon$, $\Box(Y, \P) = 0$. Since $\P$ is $\Box$-closed, $Y \in \P$.
\end{proof}

\begin{proposition}[Pyramid metric and compactification]
\label{prop:PyramidMetric}
    If $\TB \subset \L$, a metric $\rho$ on $\Pi_\L$ is defined by
    \begin{equation*}
        \rho(\P, \Q) \coloneqq \sum_{N=1}^\infty \frac{1}{2N \cdot 2^N} \cdot \haus{\Box}(\DM(\P; N, N), \DM(\Q; N, N))
    \end{equation*}
    for $\P, \Q \in \Pi_\L$. The map $\iota_\L\colon (\D/\L, \dconc) \to (\Pi_\L, \rho)$ defined by $\iota_\L(X) \coloneqq \P_X$ is a $1$-Lipschitz embedding, and $(\Pi_\L, \rho)$ is a compactification of $(\D/\L, \dconc)$.
\end{proposition}
\begin{proof}
    \textit{Metric.} We verify that $\rho$ is a metric. It suffices to show $\rho(\P,\Q)=0 \Rightarrow \P\subset\Q$, since the reverse inclusion follows by symmetry. Assume $\rho(\P,\Q)=0$, so $\s_\P = \s_\Q$. Consider the constant sequence $\P_n \coloneqq \P$, which converges weakly to $\Q$ by condition~(3) of \Cref{prop:WeakHausdorffMetrization}. For any $X\in\P$, the constant sequence $X_n \coloneqq X \in \P_n$
    $\Box$-converges to $X$, so \Cref{def:WeakHausdorffConvergence}\,(2) gives $X\in\Q$.

    \textit{Compactness.} We show that $(\Pi_\L,\rho)$ is compact. Let $\{\P_n\}_{n=1}^\infty$ be any sequence in $\Pi_\L$. Since $(\D/\L,\Box)$ is complete and separable, \Cref{lem:weakHausdorffSubSequence} yields an extraction $\iota$ and a closed set $\P$ such that $\P_{\iota(n)}$ converges weakly to $\P$. Since $\TB\subset\L$, in particular $\B\subset\L$, \Cref{theorem:BoundedLPyramidWeakLimitIsAnLpyramid} implies $\P\in\Pi_\L$. Hence $(\Pi_\L,\rho)$ is sequentially compact, thus compact.

    \textit{Embedding.} Take $X_n, X \in \D/\L$ and consider $n \to \infty$. By \Cref{thm:StaircaseEmbedding}, $X_n \to X$ in $\dconc$ is equivalent to $\s_{X_n} \to \s_X$. By \Cref{prop:WeakHausdorffMetrization}, $\s_{X_n} = \s_{\P_{X_n}} \to \s_{\P_X} = \s_X$ is equivalent to $\P_{X_n}$ converging to $\P_X$ in the weak Hausdorff sense. Thus $\iota_\L$ is a topological embedding. Moreover, by \Cref{prop:StaircaseLipschitz}, $\rho(\P_X,\P_Y)\leq\dconc(X,Y)$, so $\iota_\L$ is $1$-Lipschitz.

    \textit{Density.} For any $\L$-pyramid $\P$, by separability of $(\D/\L, \Box)$, there exists a dense sequence $\{Y_n\}_{n=1}^\infty$ in $\P$. Set $Z_1 \coloneqq Y_1$, and recursively define $Z_n \in \P$ for $n \geq 2$ by
    \begin{equation*}
        Y_n \preceq Z_n, \quad Z_{n-1} \preceq Z_n,
    \end{equation*}
    using condition~(2) of the definition of $\L$-pyramid. We show that $\P_{Z_n}$ converges to $\P$ weakly Hausdorff. For any $X \in \P$ and $\epsilon > 0$, there exists $N \in \{1,2,\ldots\}$ such that $\Box(X, Y_N) < \epsilon$. For any integer $n \geq N$, since $Y_N \preceq Z_n$, we have $Y_N \in \P_{Z_n}$, so $\Box(X, \P_{Z_n}) \leq \Box(X, Y_N) < \epsilon$. Hence $\lim_{n \to \infty} \Box(X, \P_{Z_n}) = 0$. For any $X \in \D/\L$ and extraction $\iota$ such that $\lim_{n \to \infty} \Box(X, \P_{Z_{\iota(n)}}) = 0$, we have $\Box(X, \P) = 0$, so $X \in \P$ by the closedness of $\P$. Thus $\P_{Z_n}$ converges to $\P$ weakly Hausdorff, so $\iota_\L(\D/\L)$ is dense in $(\Pi_\L, \rho)$.
\end{proof}

\begin{remark}
    If $\L$ does not contain $\TB$, then $(\Pi_\L, \rho)$ may not be a compactification of $(\D/\L, \dconc)$. For example, let $\L = \T$ and for $n = 1, 2, \ldots$, set
    \begin{equation*}
        X_n \coloneqq \left(\{0, n\}, \L \circ \{\id_{\{0,n\}}, b_1\}, \frac{\delta_0 + \delta_n}{2}\right).
    \end{equation*}
    Let $\P$ be the weak Hausdorff limit of $\P_{X_n}$. Since $X_1 \in \P_{X_n}$ for all $n$, we have $X_1 \in \P$, so $\P \neq \emptyset$.

    For any $X \in \P$, by weak Hausdorff convergence, there exists $Y_n \in \P_{X_n}$ converging to $X$ in $\Box$. For each $n$, $\od(Y_n; -1/3)$ takes value $n$ or $1$. By $\Box$-convergence, $\{\od(Y_n; -1/3)\}$ converges to $\od(X; -1/3)$, so for large $n$, $\od(Y_n; -1/3) = 1$. By the definition of $X_n$, $Y_n = X_1$, so $X = X_1$. However, for $n \geq 2$,
    \begin{equation*}
        \od(X; -1/3) = 1 < 2 \leq \od(X_n; -1/3),
    \end{equation*}
    so $\{X_n\}$ does not concentrate to $X_1$. Thus $\iota_\L$ is not a topological embedding.
\end{remark}

\subsection{The Case $\L = \lip1(\R)$}

In \Cref{prop:Lip1Extractable}, we showed that $\L$ is extractable on $(0, 1/2)$ whenever $\T \subset \L$. For $\L = \lip1(\R)$, the following result from \cite{yokota2024obsdiam} allows us to extend extractability to all of $(0,1)$.

\begin{lemma}[{\cite[Theorem~1.2]{yokota2024obsdiam}}]\label{lem:lip1minRpartdiam}
    Let $\alpha \in (0,1)$, $R > 0$, and $\mu$ be a Borel probability measure on $\R$. Then there exists $g \in \lip1(\R)$ with values in $[-R/\alpha, R/\alpha]$ such that
    \begin{equation*}
        \pd(g_*\mu;\, \alpha) = \min\{R,\, \pd(\mu;\, \alpha)\}.
    \end{equation*}
\end{lemma}

\begin{proposition}\label{prop:Lip1FullyExtractable}
    $\lip1(\R)$ is extractable on $(0,1)$.
\end{proposition}
\begin{proof}
    Set the extraction estimate $\phi(\kappa, r) \coloneqq r/(1-\kappa)$. Take any $\kappa \in (0,1)$, $\epsilon, r > 0$, and Borel probability measure $\mu$ on $\R$. By \Cref{lem:lip1minRpartdiam} applied with $\alpha = 1-\kappa$ and $R = r$, there exists $g \in \lip1(\R)$ with values in $[-r/(1-\kappa),\, r/(1-\kappa)]$ such that
    \begin{equation*}
        \pd(g_*\mu;\, 1-\kappa) = \min\{r,\, \pd(\mu;\, 1-\kappa)\}.
    \end{equation*}
    Since $g$ takes values in $[-\phi(\kappa,r),\, \phi(\kappa,r)]$, we have $(b_{\phi(\kappa,r)})_* g_* \mu = g_* \mu$. Since $\pd(\mu;\, 1-(\kappa+\epsilon)) \leq \pd(\mu;\, 1-\kappa)$, we obtain
    \begin{align*}
        \min\{r,\, \pd(\mu;\, 1-(\kappa+\epsilon))\}
        &\leq \min\{r,\, \pd(\mu;\, 1-\kappa)\}\\
        &= \pd\!\left((b_{\phi(\kappa,r)})_* g_* \mu;\, 1-\kappa\right)\\
        &\leq \pd\!\left((b_{\phi(\kappa,r)})_* g_* \mu;\, 1-\kappa\right) + 2\epsilon,
    \end{align*}
    which verifies the definition of extractability.
\end{proof}

\begin{corollary}\label{cor:Lip1OdiamLimit}
    If a sequence of $\lip1(\R)$-pyramids $\P_n$ converges weakly to $\P$, then for any $\kappa \in (0,1)$,
    \begin{align*}
        \od(\P; -\kappa)
        &= \lim_{\epsilon \to +0} \liminf_{n \to \infty} \od(\P_n; -(\kappa + \epsilon))\\
        &= \lim_{\epsilon \to +0} \limsup_{n \to \infty} \od(\P_n; -(\kappa + \epsilon)).
    \end{align*}
\end{corollary}
\begin{proof}
    Since $\TB \subset \lip1(\R)$, \Cref{prop:WeakHausdorffMetrization} gives $(1)\Leftrightarrow(3)$, so weak convergence is equivalent to $\s_{\P_n} \to \s_\P$. By \Cref{prop:Lip1FullyExtractable}, $\lip1(\R)$ is extractable on $(0,1)$, so the conclusion follows from \Cref{prop:OdiamLimit}.
\end{proof}

\section{Application to mm-Spaces}

In this section, we apply our compactification to the embedding of mm-spaces into geometric data sets proposed by Hanika et al., and show that it preserves the observable diameter.

\subsection{Compactification of Hanika--Schneider--Stumme's Embedding}

To compute the observable diameter efficiently, Hanika et al.\ proposed adopting as features the family of distance-to-point functions,
\begin{equation*}
    F_\circ(X) \coloneqq \{x \mapsto d_X(x,y) \mid y \in X\},
\end{equation*}
for an mm-space $X$~\cite[Definition~3.2]{hanika2022gds}. We set $X_\circ \coloneqq (X, F_\circ(X), \mu_X)$. With this choice, for a finite mm-space $X$ with $n = \#X$, the observable diameter $\od(X_\circ; -\kappa)$ can be computed in $O(n^3)$ time~\cite[§6.1.1]{hanika2022gds}.

More generally, fix a monoidal family $\L$ with $\T \subset \L$. For an mm-space $X$, define
\begin{equation*}
    \D^\L_\circ(X) \coloneqq (X, \L \circ F_\circ(X), \mu_X) = (X, \{x \mapsto p(d_X(x,y)) \mid p \in \L,\, y \in X\}, \mu_X).
\end{equation*}

\begin{corollary}\label{cor:DcircIsLCompact}
    Assume $\T \subset \L$. For any mm-space $X$, the geometric data set $\D^\L_\circ(X)$ is $\L$-compact. In particular, $\D^\L_\circ(X) \in \D/\L$.
\end{corollary}
\begin{proof}
    Since
    \begin{align*}
        d_X(x,y) &= |(\id_\R\circ d_X(x,\cdot))(x) - (\id_\R\circ d_X(x,\cdot))(y)| \leq d_{F_{\D^\L_\circ(X)}}(x,y)\\
        &\leq \sup_{p\in\L, z\in X} |p(d_X(z,x)) - p(d_X(z,y))| \leq d_X(x,y),
    \end{align*}
    we prove that $\D^\L_\circ(X)$ is a geometric data set. The family $\L\circ F_\circ(X)$ is $\L$-closed since $\L$ is monoidal. By \Cref{theorem:LClosedIsLCompact}, $\D^\L_\circ(X)$ is $\L$-compact.
\end{proof}

\begin{proposition}[Observable diameter under Hanika--Schneider--Stumme's embedding]
\label{prop:OdUnderHSS}
    For any mm-space $X$ and $\kappa \in (0,1)$, we have
    \begin{equation*}
        \od(\D^\L_\circ(X); -\kappa) = \od(X_\circ; -\kappa).
    \end{equation*}
\end{proposition}
\begin{proof}
    Since $\id_\R \in \L$, we see that $F_{X_\circ} \subset F_{\D^\L_\circ(X)}$, which gives $\geq$. For the reverse inequality, note that for any $f \in F_{X_\circ}$ and $p \in \L \subset \lip1(\R)$, we have
    \begin{equation*}
    \pd(p_* f_* \mu_X;\, 1-\kappa) \leq \pd(f_* \mu_X;\, 1-\kappa),
    \end{equation*}
    which gives $\leq$. This completes the proof.
\end{proof}

\begin{corollary}[Compactification of Hanika--Schneider--Stumme's embedding]
    If $\TB \subset \L$, then the closure of $\Pi^\L_\circ \coloneqq \{\P_{\D^\L_\circ(X)} \mid X \in \X\}$ with respect to $\rho$ provides a compactification of $(\D^\L_\circ(\X), \dconc)$, where $\D^\L_\circ(\X) \coloneqq \{\D^\L_\circ(X) \mid X \in \X\}$.
\end{corollary}
\begin{proof}
    By \Cref{cor:DcircIsLCompact} (note $\TB\subset\L$ implies $\T\subset\L$), the map $X\mapsto\D^\L_\circ(X)$ takes values in $\D/\L$, so $\Pi^\L_\circ\subset\Pi_\L$. By \Cref{prop:PyramidMetric}, $(\Pi_\L,\rho)$ is a compactification of $(\D/\L,\dconc)$. The closure of $\Pi^\L_\circ$ in the compact space $(\Pi_\L,\rho)$ is compact and contains $\iota_\L(\D^\L_\circ(\X))$, hence provides a compactification of $(\D^\L_\circ(\X),\dconc)$.
\end{proof}

In particular, when $\L = \lip1(\R)$, \Cref{cor:Lip1OdiamLimit} yields the complete limit formula for the observable diameter of the compactification of $\D^{\lip1(\R)}_\circ(\X)$.

\begin{acknowledgment}
    The author would like to thank Professor Takashi Shioya for many helpful suggestions, guidance, and support beyond mathematics.
\end{acknowledgment}

\bibliographystyle{abbrv}
\bibliography{all}
\end{document}